\documentclass[reqno]{amsart}
 \usepackage{amssymb,amsmath,amscd,graphicx,color,epstopdf,mathtools}
\usepackage[all]{xy}
\usepackage[all]{xy}
\usepackage{color}
\usepackage{hyperref}
\usepackage{enumerate}
\usepackage{amsthm}
\usepackage{epsfig}
\usepackage[english]{babel}

\let\oldmarginpar\marginpar
\renewcommand\marginpar[1]
{\oldmarginpar{\tiny\bf \begin{flushleft} #1 \end{flushleft}}}

\newcommand{\SL}{\mathrm{SL}}
\newcommand{\M}{\mathrm{M}}
\newcommand{\SO}{\mathrm{SO}}

\newcommand{\p}{\mathrm{p}}
\newcommand{\G}{\mathrm{G}}
\newcommand{\K}{\mathrm{K}}
\newcommand{\A}{\mathrm{A}}
\newcommand{\B}{\mathrm{B}}
\newcommand{\N}{\mathrm{N}}
\newcommand{\U}{\mathrm{U}}
\newcommand{\W}{\mathrm{W}}
\renewcommand{\M}{\mathrm{M}}
\renewcommand{\sl}{\mathrm{sl}}

   \renewcommand{\a}{\alpha}

  \newcommand{\w}{\omega}

   \newcommand{\R}{\mathbb{R}}
   \newcommand{\C}{\mathbb{C}}

  \newcommand{\LU}{\mathrm{L}\mathrm{U}}
  \renewcommand{\L}{\mathrm{L}}

     \newcommand{\cB}{\mathcal{B}}
     \newcommand{\cC}{\mathrm{C}}

    \newcommand{\cL}{\mathcal{L}}

  \newcommand{\cO}{\mathcal{O}}

    \newcommand{\cT}{\mathcal{T}}
  \newcommand{\cU}{\mathcal{U}}

\renewcommand{\gg}{\mathfrak{g}}        % Lie algebra
\renewcommand{\aa}{\mathfrak{a}}
\newcommand{\mm}{\mathfrak{m}}

\newcommand{\kk}{\mathfrak{k}}          % Lie subalgebra
\newcommand{\pp}{\mathfrak{p}}          % Factor in Cartan decomposition
\newcommand{\uu}{\mathfrak{u}}
\newcommand{\nn}{\mathfrak{n}}

\newtheorem{proposition}{{\bf Proposition}}
\newtheorem{lemma}{{\bf Lemma}}
\newtheorem{theorem}{{\bf Theorem}}
\newtheorem{corollary}{{\bf Corollary}}
\newtheorem{definition}{ {\bf Definition}}

\newtheorem{remark}{Remark}

\begin{document}

\author{David Mart\'inez Torres}
\address{Departamento de Matemtica, PUC-Rio, R. Mq. S. Vicente 225, Rio de Janeiro 22451-900, Brazil}
\email{dfmtorres@gmail.com}

\author{Carlos Tomei}
\address{Departamento de Matemtica, PUC-Rio, R. Mq. S. Vicente 225, Rio de Janeiro 22451-900, Brazil}
\email{carlos.tomei@gmail.com}

% \author{David Mart\'inez Torres}
% 
% \address{PUC-Rio de Janeiro\\
% Departamento de Matem\'atica \\ Rua Marqu\^es de S\~ao Vicente, 225\\
% G\'avea - 22451-900, Rio de Janeiro, Brazil }
% 
% \email{dfmtorres}
% 
% \author{David Mart\'inez Torres and Carlos Tomei}

% \address{PUC-Rio de Janeiro\\
% Departamento de Matem\'atica \\ Rua Marqu\^es de S\~ao Vicente, 225\\
% G\'avea - 22451-900, Rio de Janeiro, Brazil }
%
% \email{dfmtorres@gmail.com}

\title{An atlas adapted to the Toda flow}

\begin{abstract}
We describe an atlas adapted to the Toda flow on the manifold of full flags of any non-compact real semisimple Lie algebra, and on its
Hessenberg-type submanifolds. We show that in these local coordinates the Toda flow becomes linear.
The local coordinates are used to show that the Toda flow on the manifold of full flags is Morse-Smale,
which generalizes the result for traceless matrices in \cite{SV} to arbitrary non-compact real semisimple Lie algebras.
As a byproduct we describe  new features
of classical constructions in matrix theory.

\end{abstract}
\maketitle

\section{Introduction}

The non-periodic Toda lattice is a Hamiltonian model for a wave propagation along n particles in a line proposed by Toda
\cite{T}. A change of variables introduced by Flaschka \cite{F} transforms the original O.D.E. into
the matrix differential equation
\begin{equation}\label{eq:Toda}
 X'=[X,\pi_\kk X] = \cT(X),
\end{equation}
where $X$ runs over Jacobi matrices and $\pi_\kk$ is the first projection
associated to the decomposition of a matrix into its antisymmetric and upper triangular summands.

From a mathematical viewpoint (\ref{eq:Toda}) is a vector field everywhere defined on the Lie algebra of real traceless
matrices. Since it is in Lax form  it is tangent to every adjoint orbit and, in particular, to the orbit made of traceless
matrices of any fixed simple real spectrum (Jacobi matrices have simple real spectrum). Moreover, formula (\ref{eq:Toda}) implies that the Toda
vector field $\cT$ is tangent to any vector subspace which is stable upon taking Lie bracket with antisymmetric matrices.
For instance, one may pass to a compact setting by intersecting the adjoint orbit with the subspace of symmetric matrices, that is,
to the
manifold of real full flags $\cO$.
One can also intersect the adjoint orbit --or the manifold of real
full flags-- with subspaces, specified by certain zero entries, which generalize the Hessenberg property
(Jacobi matrices have the Hessenberg property, i.e., they have trivial entries below the subdiagonal).

From a Lie theoretic viewpoint, there is no reason to stick to $\sl(n,\R)$. If $\gg$ is a non-compact real semisimple Lie algebra,
a choice of Iwasawa
decomposition allows us to define the Toda vector field on $\gg$, as in (\ref{eq:Toda}) \cite{MP}. In this generality the Toda vector field
is also tangent to the regular hyperbolic adjoint orbits, to the manifold of real full flags $\cO$, and to its  Hessenberg-type submanifolds.

The main result of this paper is the construction  of an atlas for $\cO$ and its Hessenberg-type submanifolds
which linearizes the Toda vector field $\cT$.

\begin{theorem}\label{thm:atlas} Let $\gg$ be a non-compact real semisimple Lie algebra. Then there exists an atlas of the manifold of real full
flags of $\gg$ whose charts, indexed by the Weyl group, satisfy the following properties.
\begin{enumerate}[(i)]
\item The domain of each of the local coordinates defined by the atlas is a dense subset of the manifold of real full flags $\cO$ of $\gg$.
The image of the local coordinates is the whole Euclidean space.
 \item On each of the local coordinates, $\cT$ corresponds
 to a linear vector field defined on the whole Euclidean space. In particular, $\cT$ is complete in the domain of each of the local coordinates.
 \item The atlas is adapted to any Hessenberg-type submanifold: on local coordinates, the intersection with a Hessenberg-type submanifold corresponds to a vector subspace,
 and, therefore, $\cT$
 also corresponds to a linear vector field defined on the whole linear subspace.
\end{enumerate}
\end{theorem}

The result extends constructions originating from Moser's \cite{Mo} {\it inverse variables for Jacobi matrices},
which he attributes to Stieltjes. Jacobi matrices are real, symmetric matrices with strictly positive entries
in the off-diagonal positions $(i, i+1)$. The inverse variables associated with a Jacobi matrix $J$
is the list of its ordered eigenvalues (which are necessarily distinct), together with the first coordinates
of its normalized eigenvectors (which may always be taken to be strictly positive numbers).
Such diffeomorphism essentially linearizes $\cT$: eigenvalues stay put, and the evolution of the vector $c(t)$ of
first coordinates is the normalization of $exp(t \Lambda) c(0)$, where $\Lambda$ is a diagonal
matrix with the ordered eigenvalues on its diagonal. The Hamiltonian nature of the equation led to further investigation on the subject,
but there is a price to pay: the asymptotic behavior of the system (i.e., the fact that the flow converges to a diagonal matrix)
relates to a point outside of the coordinate system. More, the coordinates are very degenerate at their boundary:
the closure of the set of Jacobi matrices with a fixed spectrum is homeomorphic to a permutohedron (\cite{To}, \cite{BFR}).
However, this closure sits as a compact subset with non-empty interior of a Hessenberg-type submanifold.
Leite, Saldanha and Tomei \cite{LST1} gave a proof of item (iii) in Theorem  \ref{thm:atlas} for this simpler context:
they introduced new local coordinates on the Hessenberg-type submanifold which now contain
the relevant asymptotic points of $\cT$ in their domain  and
 where $\cT$ is linearized.
This led to a detailed study of a frequently used algorithm in numerical spectral theory,
the $\mathrm{Q}\mathrm{R}$ iteration under Wilkinson shifts (\cite{LST2},\cite{LST3}).

Our main application of Theorem \ref{thm:atlas} is to genericity properties of $\cT$ on $\cO$.

\begin{theorem}\label{thm:st-ust}  Let $\gg$ be a non-compact real semisimple Lie algebra.
 Then the unstable and stable manifolds of the $\cT$ on the  manifold of real
 full flags $\cO$ coincide with the Bruhat and opposite Bruhat cells associated to the fixed Iwasawa decomposition.
 As a consequence the Toda vector field $\cT$ is Morse-Smale.
\end{theorem}

Theorem \ref{thm:st-ust} extends to an arbitrary non-compact real semisimple 
Lie algebra the result of \cite{SV} for $\sl(n,\R)$ and $\sl(n,\C)$  and of \cite{CSS1,CSS2,CSS3} for  rank 2 non-compact semi-simple Lie algebras.

The structure of this paper is the following. In Section \ref{sec:lu} we discuss how the
unique $\LU$ factorization of special orthogonal matrices with nonzero
principal minors extends from $\SL(n,\R)$ to real semisimple Lie groups with finite center;  our analysis includes  
the interaction of the
factorization with the Weyl group action. This interaction lies at the core of the
definition of our atlas for $\cO$, which is introduced in Section \ref{sec:atlas}.
We then describe the relation of the atlas
with the Bruhat and opposite Bruhat cells. In Section \ref{sec:lin} we show that the atlas for $\cO$ linearizes
$\cT$ (item (ii) in  Theorem \ref{thm:atlas}).  The
proof is inspired in the well-known $\LU$ factorization of solutions of the Toda flow in
the manifold of full flags $\cO$ of $\sl(n,\R)$ already described forty years ago in \cite{RS} and in \cite{Sy}. In Theorem \ref{thm:st-ust},
item (ii) in   Theorem \ref{thm:atlas} is then used to identify unstable and stable manifolds
of $\cT$ with Bruhat and opposite Bruhat cells.
Section \ref{sec:Hessenberg} extends the linearization properties of the atlas to Hessenberg-type manifolds (item (iii) in Theorem \ref{thm:atlas}).
Non-compact versions of Hessenberg-type manifolds are obtained by intersecting
not just $\cO$ with the appropriate vector subspace, but the whole adjoint orbit.
In Section \ref{sec:flow} we describe a
surjective submersion with contractible fibers from  non-compact to compact Hessenberg
type submanifolds. This is accomplished by means of the flow of a vector field with appropriate normal
hyperbolicity properties.

The existence of (complete) local coordinates  where a given vector field linearizes appears to be a 
remarkable feature. The main assertion of Theorem \ref{thm:atlas}  is that 
the Toda vector field on $\cO$ has this property. Interestingly enough, already in the early eighties Duistermaat Kolk 
and Varadajan
described another  
vector field on $\cO$ with this property: it is the fundamental vector field for the adjoint action defined
by a regular hyperbolic element. However, the  \emph{Bruhat atlas}
described in \cite[Proposition 3.6]{DKV} where such fundamental vector field linearizes is different from our atlas (the open subsets 
of the two covers are different).
The Bruhat atlas also induces an atlas on Hessenberg-type submanifolds, but 
it fails to have good properties there.
% the proof that the Hessenberg-type subsets are submanifolds
% is a long and technical computation on the open subsets of the Bruhat cover \cite[Proposition 7.1]{MP}.
% This is in contrast with the elementary proof
% of the smoothness of the Hessenberg-type subsets in Lemma \ref{pro:Hess-manifold}) carried out in our new local coordinates.

 Due to the scope of this paper its methods cannot come from matrix group theory, unlike those in earlier papers 
containing instances of our results.
The methods of this paper pertain to the theory of real semisimple Lie algebras and Lie groups. When specialized to $\sl(n,\R)$ they
shed new light on how some classical constructions for matrices which are of non-linear
nature behave when we place linear constrains given by the vanishing of an appropriate set of entries.
Thus, for example, our analysis in Section \ref{sec:lu} applied to $\G=\SL(n,\R)$  proves a new feature of
the interaction of classical matrix factorization
techniques with permutations. To explain this,
recall that the projection onto the unit lower triangular factor of the $\U\L$ factorization defines
a birational morphism
\begin{equation}
\label{eq:ortho-birrat}\mathfrak{f}:\SO(n,\R)\dashrightarrow \mathrm{L}.
\end{equation}
The identity component  of the subset of matrices with nonzero principal minors, where the map is regular,
is the result of mapping $\L$ into $\SO(n,\R)$ using the Gram-Schmidt
factorization. The latter map is the inverse to the projection onto $\L$ coming from the $\L\U$-factorization.
In other words, the composition of projections associated to the Gram-Schmidt and $\U\L$-factorizations
results in an automorphism
$\Phi:\L\to \L$ encoding the comparison between $\L\U$ and $\U\L$-factorizations.
Because the latter factorizations are not everywhere defined in $\SO(n,\R)$, it is customary to factorize with (partial) pivoting a permutation
matrix. For a permutation matrix $\sigma$, comparing $\L\U$ and $\U\L$-factorizations with pivoting amounts to generalize $\Phi$ to
\begin{equation}\label{eq:lower-birat}
\Phi(\sigma):\L\dashrightarrow \L
 \end{equation}
defined by first mapping a matrix in  $\L$ into $\SO(n,\R)$ using the Gram-Schmidt
factorization, next conjugating the orthogonal matrix by $\sigma$, and then projecting
into the unit lower triangular factor of its $\U\L$ factorization.
Our main result of Section \ref{sec:lu} says the following for $\G=\SL(n,\R)$:

\begin{proposition}\label{pro:fact-Weyl-sl} Let $\sigma$ be any permutation matrix and let $\L(\sigma)\subset \L$ be
subgroup of unit lower triangular matrices whose conjugation by $\sigma$ is lower triangular.
Then the following properties hold.
\begin{enumerate}[(i)]
 \item The subgroup $L(\sigma)$  is the intersection
of $\L$ with certain coordinate subspace of the vector spaces of square matrices.
\item The restriction of $\Phi(\sigma)$ to $\L(\sigma^{-1})$ defines a real analytic diffeomorphism
    \[\Phi(\sigma): \L(\sigma^{-1})\to \L(\sigma).\]
\end{enumerate}
\end{proposition}

Another instance of how our Lie theoretic methods give new results for matrices appears in Section \ref{sec:flow}.
The theory there applied to $\gg=\sl(n,\R)$ produces a spectrum preserving  symmetrization flow with the following
additional property: for matrices with appropriately chosen zero entries, the flow keeps those entries trivial
and in the limit produces a symmetric matrix. To spell this out,  recall that an (upper)
Hessenberg matrix is a square matrix whose entries below the subdiagonal vanish. Upper triangular matrices are of
Hessenberg type, and, it is natural to consider other subspaces of the upper triangular matrices by generalizing
the Hessenberg property. This is done by choosing a \emph{profile} $\p$ of disjoint square submatrices along the subdiagonal,
and setting $V_\p$ to be the subspace of square matrices whose lower diagonal coefficients not in the chosen submatrices are zero. For
instance, if we choose no submatrices then we recover upper triangular matrices; the choice of all the $1\times 1$
submatrices gives back Hessenberg matrices. One can also regard at the intersection of any $V_\p$ with symmetric matrices, and note
the existence of an
obvious affine retraction from $V_\p$ into its intersection with symmetric matrices. However, the natural question is whether
a retraction which preserves the spectrum exists. This is yet another instance of the
problem of increasing the number of zero coefficients of a matrix while keeping its spectrum.
Our results in Section \ref{sec:flow} answer this question in the affirmative
for Hessenberg-type matrices with simple real spectrum. Inspired by \cite{AE},  we introduce the algebraic vector field
\[\mathcal{S}(X)=[X,\pi_\uu [X,X^T]],\,\, X\in \sl(n,\R),\]
where $\pi_\uu$ is the second projection associated to the decomposition of a matrix into its antisymmetric and upper triangular summands, and $X^{T}$ 
is the transpose of $X$.
 \begin{theorem}\label{thm:Hess-submersion-sl} The vector field $\mathcal{S}$ on $\sl(n,\R)$ is complete and
 its flow preserves the spectrum of a matrix and any of the Hessenberg-type
 subspaces $V_\p$.

 Its restriction to the open subset of traceless matrices with simple real spectrum has the
 following properties.
 \begin{enumerate}[(i)]
  \item Its flow lines converge to its zero set, which are symmetric traceless matrices with simple real spectrum. There, it is normally hyperbolic.
  \item The collection of flow lines with the same limit point fit into the fiber of a (smooth) submersion from
traceless matrices with simple real spectrum onto symmetric traceless matrices with simple real spectrum.
\item The submersion
restricts to a submersion upon fixing any  spectrum (fixing a regular adjoint orbit) and upon fixing any Hessenberg-type
subspace $V_\p$.
\item If a limit point belongs to a Hessenberg-type subspace $V_\p$, then its fiber lies in $V_\p$.
Equivalently, if a matrix is not in $V_{\mathrm{q}}$, then by flowing it with $\mathcal{S}$
--this including flowing infinite time to its limit--
we will never arrange the appropriate coefficients to be zero so it would belong to $V_\mathrm{q}$. In particular, the union of all flow lines converging to
diagonal traceless matrices with simple spectrum equal the upper triangular traceless matrices with simple spectrum.
 \end{enumerate}

\end{theorem}

Work by David Mart\'inez Torres was supported by Faperj and work by Carlos Tomei was supported by Capes and Faperj.

The authors are grateful to the anonimous referee for his/her thorough analysis of the manuscript and  valueble suggstions.

% MSC: 17B80; 37J35; 22E46; 14M15;  37D15. Keywords: Toda flow; Hessenberg matrices; LU factorization;
%  Iwasawa decomposition; manifold of real full flags; Bruhat cells; Morse-Smale vector field.

\section{The $\LU$ factorization, the opposite Iwasawa rulings and the action of the Weyl group}\label{sec:lu}

The $\L\U$ factorization asserts that a matrix in the special orthogonal group is the product of a unit
lower triangular matrix and an upper triangular matrix
if and only if its principal minors are nonzero. Moreover, the factorization is unique and depends smoothly (rather, analytically) on the given matrix. The same statement
holds for an $\U\L$ factorization.

Our purpose in this section is to generalize this factorization to other Lie groups and to describe its interaction with the Weyl group
action. Whereas the factorization may be known to experts, the analysis of its interaction with the Weyl group is new to the best of our knowledge. 
For the sake of our application to the Toda vector field $\cT$ we shall
opt for the $\U\L$ factorization.

Upper case will be used for Lie algebra elements, lower case for Lie group elements.
A reference for real semisimple Lie groups is \cite[Chapters VI and VII]{K}. Most
of the facts we use can be found in the much shorter
\cite[Section 2]{DKV}.

Fix once and for all the following data.
\begin{itemize}
 \item A non-compact real semisimple Lie algebra $\gg$ and a group $\G$ with finite center which integrates $\gg$; for instance,  the adjoint group of $\gg$.
 \item A Cartan involution $\Theta:\G\to \G$ with fixed point set the maximal compact subgroup $\K$. The induced involution $\theta$ on the Lie algebra gives rise to the direct sum decomposition into $+1$ and $-1$ eigenspaces
 \[\gg=\kk\oplus \pp\]
 which satisfy
 \begin{equation}\label{eq:cartan-decomposition}
[\kk,\kk]\subset \kk,\quad [\kk,\pp]\subset \pp.
 \end{equation}
 \item A maximal abelian subalgebra $\aa$ contained in $\pp$ and a root ordering yielding an Iwasawa decomposition
 at the Lie algebra and Lie group levels:
 \[\gg=\kk\oplus \aa\oplus \nn,\quad \G=\K\A\N.\]
\end{itemize}
The group $\G$ generalizes the special linear group; there, the Cartan involution is given by taking the inverse of the transpose of a matrix.
The groups $\K$, $\A$ and $\N$ generalize the special orthogonal group, the group of determinant one diagonal matrices with strictly positive entries
and the group of unit upper triangular matrices, respectively. The Iwasawa decomposition is the generalization
of the Gram-Schmidt algorithm (applied on the columns of a matrix). The generalization of the group
of upper triangular matrices with strictly positive diagonal entries is the solvable group $\U=\A\N$.
Unit lower triangular matrices generalize to the nilpotent group $\overline{\N}=\Theta\N$. The Cartan involution
applied to the Iwasawa decomposition produces the opposite Iwasawa decomposition
$\G=\K\A\overline{\N}$.

We introduce the subset which  generalizes matrices with strictly positive \footnote{To obtain a factorization result for the
analogs of matrices with nonzero principal minors one should replace $\A$ by its centralizer in $\K$.}
minors (which we view equivalently as matrices with strictly positive minors along the anti-diagonal, as we favor
the $\U\overline{\N}$ factorization).

\begin{definition}\label{def:Chevalley-big-cell}
The Chevalley big cell $\cC\subset \K$ is the image of $\N$ by the first projection of the opposite Iwasawa decomposition
\begin{equation}\label{eq:projection-to-Chevalley}  \N\subset \G=\K\A\overline{\N}\to \K.
 \end{equation}
\end{definition}

\begin{lemma}\label{lem:factorization}  The cell $\cC$ is the subset of elements of
$\K$ admitting a $\U\overline{\N}$ factorization.
The factorization is unique if and only if
(\ref{eq:projection-to-Chevalley})
is a bijection onto its image $\cC$.
\end{lemma}
\begin{proof}
Factor $g\in \N$  as $g=k(g)a(g)\bar{n}(g)=ka\overline{n}$  according to the opposite Iwasawa decomposition $\G=\K\A\overline{\N}$.
By construction, $k\in \cC$ and $k=g (a\overline{n})^{-1}$.
Since $\A$, $\overline{\N}$ and $\A\overline{\N}$ are groups, they are invariant by inversion: $ A\overline{\N}=\overline{\N}A$ and
 $a\overline{n}=\overline{n}_1a_1$. Using $\U=\A\N=\N\A$, \[k=g{a_1}^{-1}{\overline{n}_1}^{-1}\in \N\A\overline{\N}=\U\overline{\N}.\]

Conversely,
if $k\in \K$ admits a factorization
\[k=u\overline{n}=na\overline{n},\quad n\in \N,\,a\in A,\,\overline{n}\in \overline{\N},\]
then
\[n=k(a\overline{n})^{-1}=ka_1^{-1}{\overline{n}_1}^{-1}\] is an opposite Iwasawa factorization and $k$ is the first projection of $n\in \N$. Hence
the assignment $k\mapsto n$ defined above is a right inverse to  (\ref{eq:projection-to-Chevalley}).

The first conclusion is that $\cC$ is the subset of elements in $\K$  which admit a $\U\overline{\N}$ factorization.
As for uniqueness,
if the projection (\ref{eq:projection-to-Chevalley}) is a bijection
then the first factor in the factorization $k=na\overline{n}$ is uniquely determined by $k$, and so is $a\overline{n}\in A\overline{\N}$, and then also  the $\overline{\N}$ factor in this product. Therefore the $\U\overline{\N}$ factorization is unique.
Conversely,
if the $\U\overline{\N}$ factorization of $k\in \cC$ is unique, then different elements in the fiber of $k$ in (\ref{eq:projection-to-Chevalley})
 would give rise to factorizations
$k=(na)\overline{n}=(n_2a_2)\overline{n}_2$ with different factor in $\U$, contradicting  uniqueness.
\end{proof}

To discuss the uniqueness and smoothness of the $\U\overline{\N}$ decomposition
we need a more detailed description of how $\cC$ sits in $\K$.
Let $\M\subset \M'\subset \K$ be the centralizer and normalizer of $\aa$ in $\K$, respectively. We would like to prove 
the following property.
\begin{enumerate}[(i)]
\item The restriction of the projection $\N\subset \K\A\overline{\N}\to \K$
is a diffeomorphism onto $\cC$ and $\cC$ is a section of the submersion
$\K\to \K/\M'$.
\end{enumerate}
A weaker version of (i) where $\M'$ is replaced by $\M$ follows from applying the Cartan involution to Lemma 7.1 in
\cite{DKV} with $w$ equal to the identity. We present a different proof of (i) which takes advantage of the affine geometry of the Iwasawa rulings. This
geometric manifestation of the Iwasawa decomposition  will also be at the hearth of our construction of coordinates in Section \ref{sec:atlas}.

The linear analog of (i) is the following.
\begin{enumerate}[(i)]
 \item[(ii)] The restriction of the first projection $\nn\subset \kk\oplus \aa \oplus \overline{\nn}$ is a monomorphism and its
 image is complementary to
 the Lie algebra $\mm\subset \kk$ of $\M$.
\end{enumerate}
Statement (ii) is true. Indeed, the restriction to  $\nn$ of the projection parallel to $\aa\oplus \overline{\nn}$ is the monomorphism
\[X\mapsto X+\theta X,\quad X\in \nn.\]
with image given by
the orthogonal complement of $\mm$ in $\kk$ with respect to the restriction of the Killing form (see for example \cite[Equation (3.5)]{DKV}).

There is a third statement which ``interpolates'' between (i) and (ii). Let $H\in \aa$ be a regular element in the positive Weyl chamber
and let $\cO^\G$ be the adjoint orbit through $H$.
The orbit $\cO^\G$ has two preferred rulings. The corresponding affine subspaces which contain  $H$  are
\[H+\nn=\{H^n\,|\,n\in \N\},\quad H+\bar{n}=\{H^{\bar{n}}\,|\, \bar{n}\in \overline{\N}\},\]
where an element of $\G$ used as a superscript stands for conjugation by that element.
Both affine subspaces are contained in $\cO^\G$ and the action of $\K$ spreads them into rulings of $\cO^G$ -- the Iwasawa ruling and the opposite
Iwasawa ruling (see \cite{M} for details on the Iwasawa rulings).
The third statement alluded to is
\begin{enumerate}
 \item[(iii)] The Iwasawa rulings are everywhere complementary. More specifically, given a fiber
 in the Iwasawa ruling and a fiber in the opposite Iwasawa ruling, their intersection is either empty or a point.
\end{enumerate}

\begin{proposition}\label{pro:comp-rulings} The statements (i), (ii) and (iii) above are equivalent.
\end{proposition}
\begin{proof}

Let $\phi:\N\to \K/\M'$  be the composition of (\ref{eq:projection-to-Chevalley}) with the submersion $\K\to \K/\M'$.
The manifold of full flags \[\cO=\{H^k\,|\,k\in \K\}\]
is a section to both rulings (and thus reduces the structural group from affine to vector bundle). Let $\cO^\G\to \cO$ be the bundle
projection with respect to the opposite Iwasawa ruling.  The composition of the diffeomorphism $\N\to H+\nn$, $n\mapsto H^n$, with
$\cO^\G\to \cO$ is the map $n\mapsto H^{k(h)}$, where we are considering the first component of the opposite Iwasawa
factorization $\G=\K\A\overline{\N}$. Therefore its composition with the diffeomorphism $\cO\to \K/\M$, $H^{k\M}\mapsto k\M$, followed
by the projection $\K/\M\to \K'/\M'$ is exactly $\phi$. Thus (i) holds if and only if $\phi$ is a diffeomorphism over its image

The latter assertion holds if and only if
$H+\nn$ is a section to $\cO^\G\to \cO$ which misses the fibers in the orbit of $H+\overline{\nn}$ under the action
of the Weyl group $\W=\M'/\M$. Since $H+\nn$ and all fibers of $\cO^\G\to \cO$
are affine subspaces of $\gg$, this is equivalent to $H+\nn$ meeting every opposite fiber either in  a point or in the empty set, where
the latter case includes the Weyl group translates of $H+\overline{\nn}$.
The adjoint action of $\K$ on $\gg$ is linear, and,  thus,  affine. Since it preserves the opposite Iwasawa ruling and spreads the fiber
$H+\nn$ into the Iwasawa ruling, both rulings are complementary if and only if $H+\nn$ is complementary to the fibers of the opposite ruling.
Therefore statement (i) implies statement (iii).

To prove that (iii) implies (i) it remains to show that the fiber $H+\nn$ does not intersect the fiber
$H^w+\overline{\nn}^w$, where $w$ is a nontrivial element in the Weyl group\footnote{We shall abuse notation and write $w$
also for a representative in $\M'$ of the class $w\in \M'/\M$. This choice will not affect the results because $w$ will be used
to conjugate Lie subgroups or subalgebras which are either centralized or normalized by $\M$.}. Both fibers are affine spaces and therefore
its intersection must be an affine space with associate vector space the subalgebra $\nn(\overline{w})=\overline{\nn}^w\cap\nn$. The dimension
of $\nn(\overline{w})$ --the number of positive roots that $w^{-1}$ takes to negative roots-- is also the length of $w$ \cite[Lemma 2.2]{BGG}.
Since $w$ is nontrivial, its dimension is positive and
  $\nn(\overline{w})$ is also nontrivial,  contradicting (iii).

To check  (iii) at $H$ we must show that $H+\nn$ and $H+\bar{\nn}$  intersect just
at $H$. Since $\bar{\nn}$ is the tangent space to the opposite fiber at $H$, statement (iii) at $H$ is equivalent to the restriction
to $\nn$ of  the differential
of the  projection $\cO^\G\to \cO$ being a monomorphism:
$\nn\subset T_H\cO^\G\to T_H\cO$.
Under the canonical identifications $T_H\cO^\G\cong \gg/(\mm\oplus \aa)$ and $T_H\cO\cong \kk/\mm$, the restriction to $\nn$
of the composition of the quotient map $\gg\to \gg/(\mm\oplus \aa)$ with the differential, is the restriction to $\nn$ of the
projection of $\gg$ onto $\kk$ parallel to $\aa\oplus \nn$, followed by the quotient map $\kk\to \kk/\mm$.
Therefore (iii) at $H$ is equivalent to (ii).

Conversely,  we will show that (iii) at another point $H+\nn$ is again (ii) but for a different (conjugated)
Iwasawa decomposition. Let $X\in H+\nn$. We can write $X=H^n$ for a unique $n\in \N$.
The point $H^n\in H+\nn$ belongs to the opposite fibre  $H^{k(n)}+{\bar{\nn}}^{k(n)}$ based at $H^{k(n)}\in \cO$,
where we use as usual the opposite
Iwasawa factorization $\K\A\overline{\N}$.
We want to show that
\begin{equation}\label{eq:transverse}
 (H+\nn)\cap (H^n+\bar{\nn}^{k(n)})=H^n.
\end{equation}
Conjugate the fixed Iwasawa decomposition and its opposite one by $n$ to obtain
$\G=\K^n\A^n\overline{\N}^n$,  ${\N^n}=\N$.
Observe that now $H^n\in \aa^{n}$. Its adjoint orbit is still $\cO^\G$ but the compact one changes to ${\cO}^n$.
At $H^n$ the fibers of the conjugated Iwasawa opposite conjugated Iwasawa rulings are
\[H^n+\nn^n=H^n+\nn=H+\nn,\]
\[ H^n+\overline{\nn}^n=H^n+{\bar{\nn}}^{k(n)a\overline{n}(n)}=H^n+{\bar{\nn}}^{k(n)}.\]
These are the affine subspaces in (\ref{eq:transverse}). Therefore (\ref{eq:transverse}) is a consequence of (ii) applied
to $\K^n\A^n \N$ and its opposite decomposition $\K^n\A^n\overline{\N}^n$.
\end{proof}

\begin{corollary}\label{cor:smooth-factorization} The big Chevalley cell $\cC\subset \K$ is a real analytic immersed
submanifold with the following properties.
\begin{enumerate}
 \item Its projection is a diffeomorphism onto an open subset of $\K/\M'$.
\item It consists
of the elements in $\K$ which posses a $\U\overline{\N}$ or, equivalently, a $\overline{\N}\U$ factorization and thus it is closed under taking inverses.
The factorization is unique,
the factors depend in a
real analytic fashion on $\cC$, and the assignment
\begin{equation}\label{eq:un-diffeo}
 f:\cC\to \overline{\N},\quad k\mapsto \overline{n}(k),
\end{equation}
 where $\overline{n}(k)$ is the unique $\overline{N}$-component of $k$ in its $\overline{\N}\U$ factorization,
 is a real analytic diffeomorphism.
\end{enumerate}
\end{corollary}

\begin{proof}
Item (1) follows from Proposition \ref{pro:comp-rulings} and the validity of statement (ii).

If $k=(na)\overline{n}$ is its $\overline{\N}\U$ factorization, then since the Cartan involution 
 $\Theta:\G\to \G$ is the identity on $\K$, interchanges $\N$ and $\overline{\N}$ maps $\A$ onto itself, then
 \[k=\Theta(k)=\Theta(n)\Theta(a)\Theta(\overline{n})\in \overline{\N}\U.\]
Therefore $\cC$ can be equivalently defined as the subset of elements
of $\K$ which admit a $\overline{\N}\U$ factorization. Because the inversion interchanges  $\U\overline{\N}$ and $\overline{\N}\U$, we deduce
that $\cC$ is closed under taking inverses. 

%  
%  takes the Iwasawa decomposition to its opposite one and takes $\N$ to $\overline{\N}$:
% \[ \Theta( \N\subset \K\A\overline{\N})= \overline{\N}\subset \K\A\N.\]
% 
% The Cartan involution $\Theta:\G\to \G$ takes the Iwasawa decomposition to its opposite one and takes $\N$ to $\overline{\N}$:
% \[ \Theta( \N\subset \K\A\overline{\N})= \overline{\N}\subset \K\A\N.\]
% Since it is the identity on $\K$  the first projection for both factorizations are intertwined by $\Theta$ and the identity. Therefore $\cC$ is also the subset of elements
% of $\K$ admitting a $\overline{\N}\U$ factorization. The inversion interchanges the $\U\overline{\N}$ and $\overline{\N}\U$ factorizations and thus
% $\cC\subset \K$ is symmetric.

The uniqueness and real analyticity of the $\U\overline{\N}$ factorization is a consequence of
Proposition \ref{pro:comp-rulings}, the validity of statement (ii), and Lemma \ref{lem:factorization}.
% As in the proof of Lemma \ref{lem:factorization},
The inverse to $f$ in (\ref{eq:un-diffeo}) is the restriction to $\overline{\N}\subset \G$ of
the projection onto the third factor in $\G=\A\N\K$. 

% (if the latter projection is replaced by the second projection  projection
%  followed by the inversion on $\N$, then we obtain the inverse to the composition of the first projection $\N\subset \K\A\overline{\N}$
%  followed by $f$).
\end{proof}

Our next purpose it to relate the $\U\overline{\N}$ factorization with the action of the Weyl group. To that end 
we consider the subgroups of $\N$ which integrate the subalgebras  $\nn(\overline{w})=\nn\cap \overline{\nn}^{w}$ and $\nn(w)=\nn\cap \nn^w$:
\[\N(\overline{w})=\N\cap \overline{\N}^w=\{n\in \N\,|\, n^{(w^{-1})}\in \overline{\N}\},\quad \N(w)=\N\cap \N^w=\{n\in \N\,|\, n^{(w^{-1})}\in \N\}.\]
The following factorization is well known (see for instance \cite[Lemma 2.3]{DKV}):
\begin{equation}\label{eq:nilp-factorization}
 \N=\N(w)\N(\overline{w}).
 \end{equation}
The Cartan involution applied to (\ref{eq:nilp-factorization}) gives the factorization 
$\overline{\N}=\overline{\N}(w)\overline{\N}(\overline{w})$.
\begin{definition}\label{def:fact-Weyl}
 The subsets $\cC(w),\cC({\overline{w}})\subset \cC$ are the image of
the restriction of the first Iwasawa projection to $\N(w),\N(\overline{w})\subset  \K\A\overline{\N}\to \K$, respectively. 

By applying the Cartan involution we obtain an equivalent definition of $\cC(w),\cC(\overline{w})$ as the restriction 
of the first projection to $\overline{\N}(w),\overline{\N}(\overline{w})\subset \K\A{\N}\to \K$.
\end{definition}

\begin{proposition} \label{pro:fact-Weyl} We have the equalities of subsets
\begin{eqnarray}\label{eq:factorization-Weyl}
 \nonumber\cC(\overline{w^{-1}}) &= &\{k\in \cC\,|\, k^w=u\overline{n},\,\,u\in \U,\,\,\overline{n}\in \overline{\N}(\overline{w})\},\\
 \cC(w^{-1}) &= &\{k\in \cC\,|\, k^w=u\overline{n},\,\,u\in \U,\,\,\overline{n}\in \overline{\N}(w)\}
\end{eqnarray}
\end{proposition}

\begin{proof}
Let $k\in\cC({\overline{w^{-1}}})$. By definition there exists
$\overline{n}\in \overline{\N}(\overline{w^{-1}})=\N^{w^{-1}}\cap \overline{\N}$ which factors as
$\overline{n}=kan=kn_1^{-1}a_1^{-1}$,  
where we have selected inverses of elements in $\N$ and $\A$ for the sake of easing the notation in the equations which follow.
Rewrite the equality as $k=\overline{n}a_1n_1$ and conjugate it by $w$ to obtain
\[k^w=\overline{n}^wa_1^wn_1^w.\]
By hypothesis the first factor is in $\N$. The last one is in $\N^w$.
The result of applying  the factorization (\ref{eq:nilp-factorization}) with $w^{-1}$ in place of $w$ to the subset $\N^w$ of $N$ is
\[
\N^w=\N^w(w^{-1})\N^w(\overline{w^{-1}})=(\N^\w\cap \N)(\N^w\cap \overline{\N}).
\]
Applying the above decomposition to  $n_1^w$, there exists $n_2\in \N^{w}\cap \N$ such that 
 $n_2{n_1}^w\in \N^{w}\cap \overline{\N}$.
% Let $\overline{n}_2\in  \N^w\cap \N$ such that $\overline{n}_2{n_1^{-1}}^w\in \N^w\cap \overline{\N}$.
We can rewrite
\[k^w=\left(\overline{n}^w{a_1}^w n_2^{-1}\right)\left(n_2{n_1}^w\right).\]
The first factor is in $\N\A\N=\U$ and the second in $\N^w\cap \overline{\N}=\overline{\N}(\overline{w})$, and (\ref{eq:factorization-Weyl}) follows.

Conversely, let $k\in \cC$ such that $k^w=u\overline{n}^{-1}$, $\overline{n}^{-1}\in \N^w\cap \overline{\N}$. We have two factorizations
\[u^{(w^{-1})}=k\overline{n}^{(w^{-1})},\quad u^{(w^{-1})}= n_1a_1^{-1} \in \N^{w^{-1}}\A.\]
Therefore
\[n_1=k\overline{n}^{(w^{-1})}a_1=ka_2n_2\in \K\A\N.\]
Because of the decomposition  $\N^{w^{-1}}=(\N^{\w^{-1}}\cap \N)(\N^{w^{-1}}\cap \overline{\N})$ and because $n_1^{-1}\in \N^{w^{-1}}$,
one finds $n_3$ and $\overline{n}$ such that  $n_1^{-1}=n_3\overline{n}^{-1}$, so that
\[\overline{n}=n_1n_3=ka_2n_2n_3,\]
which is the $\K\A\N$ factorization of $\overline{n}\in \N^{\w^{-1}}\cap \overline{\N}$. Thus
$k\in\cC(\overline{w^{-1}})$.

The proof of the second equality is analogous (use the opposite Iwasawa factorizations to those used to prove the first equality).
\end{proof}

Proposition  \ref{pro:fact-Weyl} for $\gg=\sl(n,\R)$ is equivalent to the real analytic automorphism of subgroups
of unit lower triangular matrices announced in the introduction.

\begin{proof}[Proof of Proposition \ref{pro:fact-Weyl-sl}]
%  Let $\mathfrak{f}:\mathrm{C}\to \overline{\N}$
% be the composition of $f$ in  (\ref{eq:un-diffeo})  with the inversion on $\overline{\N}$. This is the real analytic diffeomorphism
% defined by the $\overline{\N}\U$-factorization.
It follows from Corollary \ref{cor:smooth-factorization} that the restriction to $\overline{\N}$ of the first Iwasawa
projection $\G=\K\A\N$ followed by $f:\mathrm{C}\to \overline{\N}$ in  (\ref{eq:un-diffeo}) is a real analytic diffeomorphism
$\Phi: \overline{\N}\to  \overline{\N}$. By Proposition \ref{pro:fact-Weyl} for  each representative of $w\in \W$, the result of
applying the first Iwasawa projection to $\overline{\N}(w^{-1})$, conjugating by the representative of $w$, and composing
with $f$ is a real analytic automorphism
\[\Phi(w):\overline{\N}(w^{-1})\to \overline{\N}(w).\]
Here we used  invariance of domain to conclude
that a bijective real analytic map $\Phi(w)$ is a homeomorphism (alternatively,
one can exhibit the inverse explicitly via appropriate factorizations);
we also abused notation because the automorphism depends on the representative
chosen for $w$.

For $\G=\SL(n,\R)$ the nilpotent group $\N$ are the unit upper triangular matrices. The unit lower triangular matrices --denoted  by $\L$
in the introduction--
correspond to $\overline{\N}$ in the body of the paper. The map $\mathfrak{f}$
in (\ref{eq:ortho-birrat}) in the introduction is the map $f:\mathrm{C}\to \overline{\N}$ above.
For $\SO(n,\R)$, both the $\U\L$ and $\L\U$ factorizations  are obtained by Cramer's rule, which makes clear the birational nature of $\mathfrak{f}$
and the regularity of $\Phi$.

The Weyl group for the special linear group admits canonical representatives given by permutation matrices.
For a permutation $\sigma$, the subgroup
$\overline{\N}(w)$ corresponds to the subgroup $\L(\sigma)$ obtained by intersecting $\L$ with
its conjugation by $\sigma$. Clearly $\L(\sigma)$
amounts to setting to zero certain entries of $\L$, proving item (i).

By Proposition \ref{pro:fact-Weyl}, the not everywhere defined map
$\Phi(\sigma):\L\dashrightarrow \L$ induces a real analytic diffeomorphism
from $\L(\sigma^{-1})$ onto $\L(\sigma)$, as stated in item (ii) in Proposition  \ref{pro:fact-Weyl-sl}.
\end{proof}

\noindent {\bf Example 1.} We describe explicitly the real analytic diffeomorphism in Proposition \ref{pro:fact-Weyl-sl} for $\SL(3,\R)$ and $\sigma=(2,1,3)$.

The application of the Gram-Schmidt algorithm to the Lie algebra of strictly lower diagonal matrices produces the real analytic embedding
\begin{eqnarray}\label{eq:Gram-Schmidt}
 \L &\longrightarrow & \SO(3,\R) \nonumber\\
\begin{pmatrix}
1 & 0 & 0 \\
x & 1 & 0 \\
y & z & 1
\end{pmatrix} &\longmapsto & \begin{pmatrix}
 \frac{1}{\mathit{n_1}}  &  -
\frac{x + yz}{\mathit{n_1}\mathit{n_2}}  &  \frac {xz - y}{\mathit{n_2}}  \\
\frac{x}{\mathit{n_1}}  &  \frac {1+ y^{2}- xyz}{\mathit{n_1}\mathit{n_2}}  &  - \frac{z}{\mathit{n_2}}  \\
 \frac {y}{\mathit{n_1}}  & \frac {z+ zx^{2}}{\mathit{n_1}\mathit{n_2}}  & \frac{1}{\mathit{n_2}}
\end{pmatrix}.
\end{eqnarray}
Here $n_1,n_2\in \R$ are the norm of the first column vector and of the cross product of the first and second column vectors of the
given matrix in $\L$, respectively.

Next, upon composition with the projection onto the strictly lower triangular factor of the $\U\L$ factorization, we obtain the real analytic
diffeomorphism
\begin{eqnarray*}
 \Phi: \L &\longrightarrow & \L \\
\begin{pmatrix}
1 & 0 & 0 \\
x & 1 & 0 \\
y & z & 1
\end{pmatrix} &\longmapsto & \begin{pmatrix}
 1  & 0  &  0  \\
\frac{x+yz}{\mathit{n_2}}  &  1  &  0  \\
 \frac {\mathit{n_2}y}{\mathit{n_1}}  & \frac {z+ zx^{2}-yx}{\mathit{n_1}}  & 1
\end{pmatrix}.
\end{eqnarray*}
If we compose (\ref{eq:Gram-Schmidt}) with the conjugation by a permutation we will not obtain special orthogonal
matrices with strictly positive principal minors.
For instance, the permutation $(2,1,3)$ interchanges the first and second  elements in the diagonal, and the latter need not be different from
zero. Applying then the projection onto the $\L$ factor of the $\U\L$ factorization gives a map which is not everywhere defined,
\begin{eqnarray*}
 \Phi(2,1,3): \L &\dashrightarrow & \L \\
\begin{pmatrix}
1 & 0 & 0 \\
x & 1 & 0 \\
y & z & 1
\end{pmatrix} &\longmapsto & \begin{pmatrix}
 1  & 0  &  0  \\
\frac{\mathit{n_2}x}{-1-y^2+xyz}  &  1  &  0  \\
 \frac {\mathit{n_2}y}{\mathit{n_1}}  & \frac {z+ zx^{2}-yx}{\mathit{n_1}}  & 1
\end{pmatrix}.
\end{eqnarray*}
Indeed, at any solution of $-1-y^2+xyz=0$  --which  also corresponds to an special orthogonal matrix in (\ref{eq:Gram-Schmidt}) with zero $(2,2)$-entry--
the map $\Phi(2,1,3)$ is not defined. The strictly lower triangular matrices which remain lower triangular once conjugated by $(2,1,3)$ (its own inverse)
are those for which the entry $(2,1)$ is zero. The restriction of $\Phi(2,1,3)$ to this subspace is
\begin{eqnarray*}
 \Phi(2,1,3): \L(2,1,3) &\longrightarrow & \L(2,1,3) \\
\begin{pmatrix}
1 & 0 & 0 \\
0 & 1 & 0 \\
y & z & 1
\end{pmatrix} &\longmapsto & \begin{pmatrix}
 1  & 0  &  0  \\
0  &  1  &  0  \\
 \frac {\mathit{n_2}y}{\mathit{n_1}}  & \frac {z}{\mathit{n_1}}  & 1
\end{pmatrix}.
\end{eqnarray*}
This is a real analytic diffeomorphism, as asserted in Proposition \ref{pro:fact-Weyl-sl}.

\section{An atlas for the manifold of full flags and its relation with the Bruhat cells}\label{sec:atlas}

Let  $H\in \aa$ be a regular element in the positive Weyl chamber and let  $\cO=\{H^k\,|\,k\in \K\}$ be the manifold of real full flags. In this section   
we will use the quotient model coming from the canonical diffeomorphism $\cO\cong \K/\M$, $H^{k\M}\mapsto k\M$. Our constructions there should also 
be interpreted as constructions on $\cO$ (for which we will keep the same notation).

There exist a standard \emph{Bruhat atlas} for $\K/\M$ \cite{DKV}. 
A lesser-known open cover of $\K/\M$ was defined in \cite{FH}. Our first purpose in this section 
is to introduce coordinates on the open subsets of the later cover. We will argue that these local coordinates are  natural consequence of how the 
Iwasawa decomposition reflects on the geometry of adjoint orbits: the coordinates will be a byproduct of the 
 $\U\overline{\N}$ factorization discussed in Section \ref{sec:lu} and of the canonical diffeomorphism of (appropriate) Iwasawa fibers with $\overline{\N}$. 
% Iwasawa fibers are orbits of the . 
% Alternatively, they can also be inferred by combining the Bruhat altas together with the
% so-called companion embeddings, and it is in this guise that they may be known to some experts.
Our second purpose is to discuss the interaction of these local coordinates with the Bruhat and opposite Bruhat cells.

Let $\B=\M\A\N\subset \G$. The Bruhat cells of $\G$ are the double cosets $\B\backslash \G /\B$. 
They are parametrized by the Weyl group and produce the partition $\G=\coprod_{w\in \W}\B w\B$;
the cell parametrized by $w_0$, the longest element of $W$, is open and dense. There is an induced 
partition in Bruhat cells  
$\G/\B=\coprod_{w\in W} \B w\B$ (orbits for the action of $\B$).
Each left translate of the open dense cell contains the corresponding Bruhat cell, and thus the collection  
 $ww_0\B w_0 \B$, $w\in W$, is an open cover of $\G/\B$. 
It produces the Bruhat  cover of $\K/\M$ via the diffeomorphism
\begin{equation}\label{eq:diffeo-Iwasawa}
\G/\B\to \K/\M,\quad g\B\mapsto k(g)\M,\quad \G=\K\A\N.
\end{equation}
More explicitly, if we denote the image of $ww_0\B w_0 \B$ in $\K/\M$ by $\cL_w$, then 
\[\cL_w=k(ww_0\B w_0 \B)\M=wk(w_0\N w_0\M\A\N)\M=wk(\overline{\N})\M=w\cC\M.\]

% The Bruhat cells of $\G/\B$ are the cells of the induced partition of $\G/\B$, and they  by Bruhat cells 
% The Bruhat cells of $\G/\B$ are the 
% %   and therefore $w\Bw_0\B_w_0$, $w\in W$ is an open cover 
% of $G$ by dense subsets. If we set $\cL_w=\K\cap w\Bw_0\B_w_0$, then we obtain the Bruhat open cover of $\K$
% by the  elements of $W$ def
% \[\G=\bigcup_{w\in \W}\B w_0\B w=\bigcup_{w\in \W}\B w_0\B \w_0 w.\]
% is an open cover of 
% Each such open subset acts on $H$ giving rise to a  subset of $\cO^\G$,

The subsets we are interested in are right translates of sorts  of $\cL_e$ by $w$.

\begin{definition}\label{def:cover}
For each element of the Weyl group $w\in \W$, let $\cU_w \subset \K/\M$ be
\[\cU_w=\cC w\M=\{{kw}\M\,|\,k\in \cC\}.\]
% \[\cU_w=H^{\cC w}=\{H^{kw}\,|\,k\in \cC\}.\]
\end{definition}
The subsets $\cU_w$ were introduced  in \cite[Definition 1, Lemma 1]{FH} for $\G$ a complex semisimple Lie group.
% We  will see in Remark \ref{rem:two-covers} that 
% $\cU_w$ and $\cL_w$ are in general different. 

% To show that the  subsets $\cU_w$, $w\in \W$,  cover  $\cO$, we recall
% that  the (disjoint)  double cosets $\B\backslash \G /\B$ --the Bruhat cells of $\G$-- are parametrized by the Weyl group,
% \[\G=\coprod_{w\in \W}\B w\B.\]
% The cell parametrized by $w_0$, the longest element of $W$, is open and dense. Its right (and left) translates
% by the  elements of $W$ define a (finite) open cover of $\G$:
% \[\G=\bigcup_{w\in \W}\B w_0\B w=\bigcup_{w\in \W}\B w_0\B \w_0 w.\]
% Each such open subset acts on $H$ giving rise to a  subset of $\cO^\G$,

\begin{lemma}\label{lem:cover}
 The collection $\cU_w$, $w\in \W$, defines an open cover of $\K/\M$ by dense subsets diffeomorphic to $\cC$.  Each $\cU_w$ intersects
 the Weyl group orbit of $e\M$ in $w\M$.
\end{lemma}

\begin{proof}
To show that subsets $\cU_w$ are open, dense and fit into a cover of $\G$ we appeal to Lemma 2 in \cite{FH}. Even though the statement there is for 
complex semisimple Lie groups,
the proof relies on properties of the Bruhat decomposition which are common to arbitrary semisimple Lie group, and thus it is valid in that generality.
We can either appeal again to \cite{FH} or use results from Section \ref{sec:lu} to exhibit the (obvious) canonical diffeomorphism of
$\cC$ with $\cU_w$, and to argue
that $w\M$
is the only intersection point of $\cU_w$ and the Weyl group orbit of $e\M$: by item (1) in Corollary \ref{cor:smooth-factorization}  
$\cC w\subset \K$ is a section  to the submersion $\K\to \K/\M$ over its image $\cU_w$. By Proposition \ref{pro:comp-rulings} $\cC w\subset \K$ intersects 
the fiber $w'\M$ if and only if $w'=w$.
\end{proof}

The Bruhat cover exploits the $\overline{\N}\U$ factorization. In \cite[Lemma 7.1]{DKV} coordinates $\psi_w:\cL_w\to \overline{\nn}$ are defined 
by sending $wk$ to the 
logarithm of the component in $\overline{\N}$ of the $\overline{\N}\U$ factorization of $k\in \cC$. 

It is also the case that each subset $\cU_w\subset \K/\M\cong \cO$ supports natural coordinates associated to the $\U\overline{\N}$ factorization and the Iwasawa rulings:

\begin{definition}\label{def:atlas}
For $w\in \W$, the  local coordinates on $\cU_w\subset \K/\M$ are
 \begin{equation}\label{eq:local-coordinates}
  \varphi_w: \cU_w\to H^w+\overline{\nn},\quad kw\M\mapsto \varphi_w(kw\M),
 \end{equation}
where  $\varphi_w(kw\M)$ is the unique solution of
\[\varphi_w(kw\M)={(H^w)}^{\overline{n}(k)},\]
and $\overline{n}(k)$ is the component in  $\overline{\N}$ of the $\U\overline{\N}$ factorization of $k\in \cC$.

% For $w\in \W$, the  local coordinates on $\cU_w\subset \cO$ are
%  \begin{equation}\label{eq:local-coordinates}
%   \varphi_w: \cU_w\to H^w+\overline{\nn},\quad H^{kw}\mapsto \varphi_w(H^{kw}),
%  \end{equation}
% where  $\varphi_w(H^{kw})$ is the unique solution of
% \[\varphi_w(H^{kw})={(H^w)}^{\overline{n}(k)},\]
% and $\overline{n}(k)$ is the component in  $\overline{\N}$ of the $\U\overline{\N}$ factorization of $k\in \cC$.
\end{definition}
The map $\varphi_w$ is  the composition of
$\cU_w\to \cC$, $kw\M\mapsto k$, with $f:\cC\to \overline{\N}$, $k\mapsto \overline{n}(k)$, and with the inverse of the diffeomorphism
$\overline{\N}\to H^w+\overline{\nn}$, $\overline{n}\mapsto H^{\overline{n}w}$. By items (1) and (2) in Corollary \ref{cor:smooth-factorization},
the first and second maps are diffeomorphisms, hence $\varphi_w$ is a diffeomorphism. We interpret $\varphi_w$ as local  coordinates by regarding
the affine space $H^w+\overline{\nn}$  as a vector space with origin $H^\w$.

To argue that these coordinates are natural from the viewpoint of the geometry of the Iwasawa rulings of $\cO^{\G}$,
we regard $\cU_w$ as a subset of $\cO$:   
$\cU_w=\{H^{kw}\,|\,k\in \cC\}$. Firstly, upon acting by $\overline{n}(k)k^{-1}$ one goes from $H^{kw}\in \cU_w\subset \cO$ 
to $H^{\overline{n}(k)w}\in \cO^{G}$. Secondly, one observes that $H^{\overline{n}(k)w}$ belongs to the Iwasawa fiber\footnote{
The Iwasawa ruling is determined by the fixed root ordering; root orderings are in bijection with Weyl group elements. In our case
$H^{w}+\overline{\nn}$ is an Iwasawa fiber of the ruling defined by $w^{-1}$.} $H^{w}+\overline{\nn}$.

Next, we discuss the relation of the atlas $(\cU_w,\varphi_\w)$, $w\in W$, with the Bruhat and opposite Bruhat cells.
The diffeomorphism  (\ref{eq:diffeo-Iwasawa}) induces an action of $\G$ on $\K/\M$. The Bruhat and opposite Bruhat cells of  $\K/\M$ ---
which we denote by $\cB_w$ and $\overline{\cB}_w$, $w\in W$ ---
 are the orbits of the action of $\B$ and $\overline{\B}$, respectively. 
Both $\cB_w$ and $\overline{\cB}_w$ are contained in the corresponding open subset $\cL_w$ of the Bruhat cover and their images by the Bruhat coordinates
are vector subspaces of $\overline{\nn}$ \cite[Lemma 7.1]{DKV}. 

The following two results show that the Bruhat and opposite Bruhat cells 
have the 
the same property with respect to $(\cU_w,\varphi_w)$.

% overliThe group $\G$ acts on $\K/\M$ via the diffeomorphism
% \begin{equation}\label{eq:diffeo-Iwasawa}
% \G/\B\to \K/\M,\quad g\B\mapsto k(g)\M,\quad \G=\K\A\N.
% \end{equation}
% The Bruhat and opposite Bruhat cells are the orbits of the action of  $\B$ and $\overline{\B}$ respectively. They
% are parametrized by the Weyl group elements $w\in \W$, and  are denoted by $\cB_w$ and $\overline{\cB}_w$.

\begin{lemma}\label{lem:cover-Bruhat} The  cells at $w$, $\cB_w$ and $\overline{\cB}_w$, sit inside $\cU_w$. More precisely:
\begin{equation}\label{eq:cover-Bruhat}
\cB_w={\cC({\overline{w^{-1}}})}^{w}w\M\subset \cU_w,\quad \overline{\cB}_w={\cC({w^{-1}})}^{w}w\M\subset \cU_w.
 \end{equation}
\end{lemma}
\begin{proof}
The Bruhat cell at  $w$ in $\G/\B$ is 
\[\B w\B=\N w\M\A\N=\N(\overline{w})\N(w) w\M\A\N=\N(\overline{w})w\M\A\N=w\N(\overline{w})^{w^{-1}}\M\A\N.
\]
The Bruhat cell $\cB_w$ in $\K/\M$ is
\[k(w\N(\overline{w})^{w^{-1}})\M=(wk(\N(\overline{w})^{w^{-1}})w^{-1})w\M.\]
Since \[\N(\overline{w})^{w^{-1}}=(\N\cap \overline{\N}^w)^{w^{-1}}=\N^{w^{-1}}\cap \overline{\N}=\overline{\N}(\overline{w^{-1}}),\]
it follows from  Definition \ref{def:fact-Weyl} that $k(\N(\overline{w})^{w^{-1}})=\cC(\overline{w^{-1}})$. Therefore
% \[(wk(\N(\overline{w})^{w^{-1}})w^{-1})w=w\cC_{\overline{w^{-1}}}{w^{-1}}w={(\cC_{\overline{w^{-1}}})}^{w}w.\]
\[\cB_w={\cC(\overline{w^{-1}})}^{w}w\M.\]
The equality $\overline{\cB}_w={\cC({w^{-1}})}^{w}w\M$ is proven in an analogous manner.

The equalities in (\ref{eq:factorization-Weyl}) imply the following inclusion of subsets of $\K$:
\[{\cC({\overline{w^{-1}}})}^{w}, {\cC({w^{-1}})}^{w}\subset \cC.\] 
Since $\cC w$ is a section to $\K\to \K/\M$ we obtain inclusions 
of subsets of $\K/\M$
\[{\cC({\overline{w^{-1}}})}^{w}w\M,{\cC({w^{-1}})}^{w}w\M\subset \cC w\M=\cU_w,\]
which gives the desired result.
% % the subsets in the right hand sides of  (\ref{eq:cover-Bruhat}) lie in $\cU_w$.
% The Bruhat cell at  $w$ in $\G/\B$ is 
% \[\B w\B=\N w\M\A\N=\N(\overline{w})\N(w) w\M\A\N=\N(\overline{w})w\M\A\N=w\N(\overline{w})^{w^{-1}}\M\A\N.
% \]
% Therefore the Bruhat cell $\cB_w$ in $\K/\M$ is
% \[k(w\N(\overline{w})^{w^{-1}})\M=(wk(\N(\overline{w})^{w^{-1}})w^{-1})w\M.\]
% Since \[\N(\overline{w})^{w^{-1}}=(\N\cap \overline{\N}^w)^{w^{-1}}=\N^{w^{-1}}\cap \overline{\N}=\overline{\N}(\overline{w^{-1}}),\]
% it follows from  Definition \ref{def:fact-Weyl} that $k(\N(\overline{w})^{w^{-1}})=\cC(\overline{w^{-1}})$. As a consequence
% \[(wk(\N(\overline{w})^{w^{-1}})w^{-1})w=w\cC_{\overline{w^{-1}}}{w^{-1}}w={(\cC_{\overline{w^{-1}}})}^{w}w.\]
% \[\cB_w={\cC(\overline{w^{-1}})}^{w}w\M\]
% The second equality is proven in an analogous manner.
\end{proof}

\begin{proposition}\label{pro:atlas-Bruhat}
For $w\in \W$, the local coordinates  $\varphi_w: \cU_w\to H^w+\overline{\nn}$  
transform the  Bruhat cells $\cB_w$ and $\overline{\cB}_w$ into affine subspaces,
\begin{equation}\label{eq:atlas-Bruhat}
\varphi_w(\cB_w)=H^w+\overline{\nn}(\overline{w}),\quad \varphi_w(\overline{\cB}_w)=H^w+\overline{\nn}(w).
\end{equation}
\end{proposition}
\begin{proof}
% By Lemma \ref{lem:cover-Bruhat} $\cB_w\subset \cU_w$. 
To calculate $\varphi_w(\cB_w)$ we use that 
by  Lemma \ref{lem:cover-Bruhat} and Proposition \ref{pro:fact-Weyl}  
\[\cB_w={\cC({\overline{w^{-1}}})}^{w}w\M\subset \K/\M, \quad  \cC(\overline{w^{-1}})=\{k\in \cC\,|\, k^w=u\overline{n},\,\,u\in \U,\,\,\overline{n}\in \overline{\N}(\overline{w})\},\]
respectively. 
%  $\cC({\overline{w^{-1}}})$  are the elements in $\cC$ whose conjugates by $w$ have $\U\overline{\N}$
%  factorization with second factor in $\overline{\N}(\overline{w})$.
 Therefore $\cB_w=\{kw\M\,|\,k=u\overline{n},\,\overline{n}\in \overline{\N}(\overline{w})\}$ and this implies
 \[\begin{split}
\varphi_w(\cB_w) &= \{{(H^w)}^{\overline{n}}\,|\,\overline{n}\in \overline{\N}(\overline{w})\}={(H^w)}^{\overline{\N}(\overline{w})}=
{(H^w)}^{\overline{\N}\cap \N^{w}}=\\
  &= (H^w+\overline{\nn})\cap (H^w+\nn^{w})=H^w+\overline{\nn}(\overline{w}).
\end{split}
\]
The second equality is proven in an analogous manner.
 \end{proof}

\noindent {\bf Example 2.} \label{ex:chart} We shall compute explicitly (the inverse of) $\varphi_w:\cU_w\subset \cO\to H^w+\overline{\nn}$ for
$\gg=\mathrm{sl}(2,\R)$ and $w=e$.

The domain of the chart is
\[H+\overline{\nn}= \begin{pmatrix}
           \lambda & 0\\ x & -\lambda
          \end{pmatrix},\]
where $\lambda>0$ is fixed and $x\in \R$.
Its image will be $\mathcal{U}_e\subset \cO$, where $\cO$ sits inside the 2-dimensional vector space of $2\times 2$ symmetric  trace zero matrices.

First, we compute the diffeomorphism $H+\overline{\nn}\to \overline{\N}$, $H+X\mapsto g$,
determined by the equality $H^g=H+X$:
\[g(X)=\begin{pmatrix}
           1 & 0\\  \tfrac{x}{2\lambda} & 1
          \end{pmatrix}.\]
Now, for $g\in \overline{\N}$ we compute the unique orthogonal matrix $k\in \K$ whose $\overline{\N}$ factor in its $\U\overline{\N}$
decomposition equals $g$: we invert $g$, compute its orthogonal factor in the Gram-Schmidt factorization and transpose this orthogonal
factor:
\[k(g)=\begin{pmatrix}
           \tfrac{1}{\sqrt{1+a^2}} &  -\tfrac{a}{\sqrt{1+a^2}}\\  \tfrac{a}{\sqrt{1+a^2}} & \tfrac{1}{\sqrt{1+a^2}}
          \end{pmatrix},\quad g=\begin{pmatrix}
           1 & 0\\  a & 1
          \end{pmatrix}.\]
Finally, we conjugate $H$ with $k$ for $k=k(g)$ and $g=g(X)$:
\begin{eqnarray}\label{eq:chart}\nonumber
 \varphi_e^{-1}: H+\overline{\nn}\cong \R &\longrightarrow & \cO\subset \pp\subset \sl(2,\R)\\
       \begin{pmatrix}
           \lambda & 0\\ x & -\lambda
          \end{pmatrix}  & \longmapsto & \frac{1}{1+\tfrac{x^2}{4\lambda^2}}\begin{pmatrix}
           \lambda-\tfrac{x^2}{4\lambda} & x\\
         x & \tfrac{x^2}{4\lambda}-\lambda
          \end{pmatrix}.
\end{eqnarray}

\begin{remark}\label{rem:two-covers} (The two open covers and the $\U\overline{\N}$ and $\overline{\N}\U$ factorizations).
The  subsets  $\cL_w$ and $\cU_w$ agree for $w=e$, but they are in general different for arbitrary $w$. 
\[\cL_w=\cU_w\ \Longleftrightarrow w^{-1}\cL_w=w^{-1}\cU_w \Longleftrightarrow  \cC\M=\cC^{w^{-1}}\M\subset \K/\M. 
\]
The rightmost equality does not hold in general.
For instance, for $\G=\mathrm{SL}(3,\R)$ and $\K=\mathrm{SO}(3,\R)$  the subset $\cC\subset \K$ consists of
matrices in $\K$ whose minors along the antidiagonal are positive; its saturation by the fibers of 
$\K\to \K/\M$ are matrices in $\K$ whose minors along the antidiagonal are non-zero. This subset is 
not closed under conjugation by the permutation $(1,3,2)$. For example, take a matrix in the subset whose  $(2,2)$-entry vanishes. 

The subsets $\cU_w$ are engineered to exploit the $\U\overline{\N}$ factorization, whereas the Bruhat cover exploits the $\overline{\N}\U$ factorization.
Because the inversion map on $\cC$ intertwines both factorizations it is natural to consider for each $w$ the diffeomorphism
\begin{equation}\label{eq:un-to-nu}
\tau_w:\cU_w\to \cL_{w^{-1}},\quad kw\M\mapsto w^{-1}k^{-1}\M. 
\end{equation}
If $\overline{n}(k)$ is the component in $\overline{\N}$ of the $\U\overline{\N}$ factorization of $k$,
then the images of $kw\M$ and $\tau_w(kw\M)$ in $\G/\B$ by the diffeomorphism 
(\ref{eq:diffeo-Iwasawa}) are 
\[\overline{n}(k)w\B,\quad w^{-1}{\overline{n}(k)}^{-1}\B.\]
%  
%  
%   is given by the action on $w\M$ of the component in $\overline{\N}$ of the $\U\overline{\N}$ factorization of $k$.
% The subsets of the Bruhat cover in turn exploit the $\overline{\N}\U$ factorization. 
% The image of $wk\M\in \cL_w$ by (\ref{eq:diffeo-Iwasawa}) is given by the action on $e\M$
% of the component in $\overline{\N}$ of the $\overline{\N}\U$ factorization of $k$, followed by the action of $w$.
\end{remark}

% \begin{remark} (The two atlases and the companion embeddings)
% The coordinates for the Bruhat atlas are defined in   \cite[Lemma 7.1]{DKV}
% \[\psi_w:\mathcal{L}_w\longrightarrow \overline{\N}\overset{\log}{\longrightarrow}\overline{\nn},\]
% where the first arrow sends $wk\M$ to the $\overline{\N}$-component of the $\overline{\N}\U$ factorization of $k$. 
% 
% The two atlases $\{\cU_w,\varphi_e\},\{\cL_w,\psi_w\}$ ,$w\in W,$ can be related by the so-called companion embeddings. They are defined as follows:
% \[\tau_w: \cU_\w \to \cL_{w^{-1}},\quad  kw\M\mapsto w^{-1}k^{-1}\cM.\]
% This diffeomorphisms were introduced in \cite{FH2} for complex semisimple Lie groups. The adjective embedding describes the original purpose of these diffeomorphisms:
% one oftens uses them to compactify the image by $\tau_w$ of appropirate submanifolds of $\cU_w$ (and this is typically done for $w=e$, so $\cU_e=\cL_e$).
%  
% 
% \end{remark}

\section{The Toda flow in local coordinates}\label{sec:lin}

Formula (\ref{eq:Toda}) describing the Toda vector field $X \mapsto \cT(X)$ is valid  for any non-compact real semisimple Lie algebra with fixed Iwasawa decomposition:
the projection $\pi_\kk$ is taken with respect to the direct sum decomposition $\gg=\kk\oplus \uu$, where $\uu=\aa\oplus \nn$ (this coincides with the definition of
the projection operator  in \cite[Section 2]{MP} by means of root spaces). The Lax form implies
that $\cT$ is tangent to every adjoint orbit, in particular to the regular hyperbolic orbit $\cO^\G$. Because
of the properties of the Cartan decomposition in (\ref{eq:cartan-decomposition}), $\cT$ is tangent to the subspace $\pp$.
Therefore $\cT$ is tangent to the manifold of real full flags $\cO$.

In this section we  show that the local coordinates in Definition \ref{def:atlas}
linearize $\cT$ on $\cO$. This, together with the result in Section \ref{sec:atlas}, will be the ingredients to prove  Theorem \ref{thm:st-ust}
relating Bruhat and opposite Bruhat cells to unstable and stable manifolds of $\cT$.

Because the Toda vector field at $X\in \cO$ is the fundamental vector field of $-\pi_\kk(X)\in \kk$ for the adjoint action of $\K$ on $\cO$,
for any of its integral curves $X(t)$ there exists a canonical choice of curve  $k(t)\subset \K$, $k(0)=e$, such that \[X(t)={X(0)}^{k(t)}.\]
Let $\mathcal{Y}_t$ be the right invariant vector field in $\K$ which at the identity equals $-\pi_\kk(X(t))$,
\[\mathcal{Y}_t(k)=dR_k(-\pi_\kk(X)),\quad k\in \K,\]
where $R_k:\K\to \K$ is given by  right multiplication by $k$.
Then $k(t)$ is the curve defined by applying the flow of $\mathcal{Y}_t$ to the identity element $e\in \K$.
 
The subset $\cC\subset \K$ is not open in general\footnote{It is open exactly when $\G$ is a split group. For instance, for $\G=\SL(n,\R)$}. The next result shows how the above curves in $\K$ are related
to $\cC$.
\begin{lemma}\label{lem:Toda-big-cell}
Let $X(t)$ be the integral curve of $\cT$ through $X(0)\in \cU_w$, $X(0)={(H^{w})}^{k_0}$, and let 
 $k(t)\subset \K$ be the curve through $e$
described above. 
Then  for small values of $t$ 
\[k(t)k_0\subset \cC m_0,\] where 
$m_0$ is the component in $\M$ of the factorization of $k_0\in \cC\M$.
 \end{lemma}
\begin{proof}

We have to prove that $k(t)k_0$ is contained in one of the leaves $\cC m$, $m\in \M$, which foliate $\cC\M$.
It suffices to show that each velocity vector
\[dR_{k_0}(k'(t))=dR_{k_0}(\mathcal{Y}_{t}(k(t)))=\mathcal{Y}_t(k(t)k_0)\]
is tangent to the underlying
distribution. This amounts to proving the following:   
If $X=\cU_w$, $X={(H^{w})}^{k_1}$, $k_1\in \K$, and $\mathcal{Y}_{X}$ denotes the right invariant vector field in $\K$ determined 
by $-\pi_\kk(X)\in \kk$, then $\mathcal{Y}_{X}(k_1)$, is tangent
to the underlying distribution.

A more explicit description of $\mathcal{Y}_X$ follows from an analyisis of the differential of the three Iwasawa projections for $\G=\K\A\N$.
Consider for any $X\in \cU_w$, $k\in \K$, the curve $\gamma_{X,k}(s)=k^{-1}\exp(sX)\subset \G$. Then 
by \cite[Lemma 2.2]{M} we have
\[
X= dL_{k^{-1}}\left(\frac{d}{ds}\bigg\rvert_{s=0}k(\gamma_{X,k}(s))\right)+ \frac{d}{ds}\bigg\rvert_{s=0}a(\gamma_{X,k}(s))
+\frac{d}{ds}\bigg\rvert_{s=0}n(\gamma_{X,k}(s)),
\]
where $L_{k^{-1}}:\K\to \K$ is given by left multiplication by $k^{-1}$ and the curves whose velocity at $s=0$ we are computing are the three Iwasawa projections of $\gamma_{X,k}(s)$.
Because the three velocity vectors belong to $\kk$, $\aa$ and $\nn$, respectively, 
the first summand is the left invariant vector field at $k^{-1}$ determined by $\pi_\kk(X)\in \kk$, and,  thus
\[ \mathcal{Y}_{X}(k)=\frac{d}{ds}\bigg\rvert_{s=0}{k(\gamma_{X,k}(s))}^{-1}.\]
Factor  $k_1=u_1\overline{n}_1m_1$, $u_1\in \U$, $\overline{n}_1\in \overline{\N}$, $m_1\in \M$, so that
\[\begin{split}
   {k(\gamma_{X,k}(s))}^{-1} &={k(k^{-1}\exp(s{(H^{w})}^{k_1}))}^{-1}={k(m_1\exp(sH^{w})k_1^{-1})}^{-1}\overline{n}_1^{-1}u_1^{-1}k=\\
    & ={k(\overline{n}_1^{-1}{\exp(sH^{w})}^{m_1\overline{n}_1}u_1^{-1})}^{-1}\overline{n}_1^{-1}u_1^{-1}k.
   \end{split}
\]
Because $A^{\overline{n}_1}\subset \overline{\N}$ and $\cC$ is closed under inversion the curve 
${k(\gamma_{X,k_1}(s))}^{-1}$ lies in  $\cC m_1$, and therefore 
$\mathcal{Y}_{X}(k_1)$
is tangent to the underlying distribution.
\end{proof}

The linearization of the Toda vector field in the local coordinates $(\cU_w,\varphi_w)$  builds on Lemma \ref{lem:Toda-big-cell} and 
on well known factorization properties of solutions of the Toda flow (see
e.g. \cite[Propositions 26,27]{RS} or \cite[Section 3]{Sy}). It also reveals a connection  of $\cT$ with
the adjoint action of $\A$ on Iwasawa fibers. 

\begin{theorem}\label{thm:lin-atlas}
 Let $\gg$ be a non-compact real semisimple Lie algebra with fixed Iwasawa decomposition and let $\varphi_w:\cU_w\to H^\w+\overline{\nn}$, $w\in \W$,
 be the local coordinates in Definition \ref{def:atlas}.
Then the diffeomorphism $\varphi_w$ transforms the Toda vector field $\cT$ into the linear vector field
 \begin{equation}
 \label{eq:Toda-linear-coord}X'=[H^{w},X],\quad X\in H^w+\overline{\nn}, 
 \end{equation}
 whose stable and unstable manifolds at its critical point $H^w$ are the subspaces
 \[H^w+\overline{\nn}(w),\quad H^w+\overline{\nn}(\overline{w}).\]
\end{theorem}

\begin{proof}
Let $X(t)$ be an integral curve of the Toda vector field 
\[X'(t)=[-\pi_\kk (X(t)),X(t)],\quad X(0)\in \cU_w,\quad X(0)={(H^{w})}^{k_0},\]
and let $k(t)\subset \K$ be the corresponding curve throught the identity
\[X(t)=X(0)^{k(t)},\quad   dR_{k_0}(k'(t))=\mathcal{Y}_t(k(t)k_0).\]
We assume that $t$ is small enough so that $X(t)\subset \cU_w$ and thus we can define $B(t)=\varphi_w(X(t))$. 
By Lemma \ref{lem:Toda-big-cell} we have a factorization
\begin{equation}\label{eq:factor-zero}
k(t)k_0=u(t)\overline{n}(t)m_0,\quad u(t)\in \U,\,\,\overline{n}(t)\in \overline{\N},\,\,m_0\in \M.
\end{equation}
Thus by Definition \ref{def:atlas}
\begin{equation}\label{eq:image-curve}
B(t)={(H^{w})}^{\overline{n}(t)}\subset H^{w}+\overline{\nn}.
\end{equation}
The equality (\ref{eq:image-curve}) implies  that $B'(t)$ is the value at $B(t)$ of the 
fundamental vector field of $dR_{\overline{n}(t)^{-1}}(\overline{n}'(t))\in \overline{\nn}$ for the adjoint
action of $\overline{\N}$ on $H^{w}+\overline{\nn}$.  This means
\begin{equation}\label{eq:Toda-sub-zero}
B'(t)=[ dR_{\overline{n}(t)^{-1}}(\overline{n}'(t)),B(t)].
\end{equation}
Next, we describe an explicit relation of $dR_{\overline{n}(t)^{-1}}(\overline{n}'(t))\in \overline{\nn}$ and the Toda vector $-\pi_\kk(X(t))$.
To that end we regard
(\ref{eq:factor-zero}) as a curve in $\G$ and differentiate it to get an equality of vectors at $T_{k(t)k_0}\G$:
\[dR_{k_0}(k'(t))=dR_{\overline{n}(t)m_0}(u'(t))+dL_{u(t)}\circ dR_{m_0}(\overline{n}'(t)).\]
Appling the differential of $L_{m_0^{-1}\overline{n}(t)^{-1}u(t)^{-1}}$ we obtain an equality of vectors at $\gg$
\[
-\pi_\kk(X(t))=dR_{u(t)^{-1}}(u'(t))+dL_{u(t)}\circ dR_{u(t)^{-1}}\circ dR_{\overline{n}(t)}(\overline{n}'(t)),
\]
to which we apply the adjoint action of $u(t)^{-1}$ to get
\begin{equation}\label{eq:sum-zero}
 -\pi_\kk(X(t))^{u(t)^{-1}}=dR_{u(t)^{-1}}(u'(t))^{u(t)^{-1}}+dR_{\overline{n}(t)}(\overline{n}'(t)).
\end{equation}
Let $\pi_{\overline{\nn}}$ denote the first projection associated to the splitting $\gg=\overline{\nn}\oplus \mm\oplus \uu$. The 
equality (\ref{eq:sum-zero}) implies that 
\begin{equation}\label{eq:sum-one}
dR_{\overline{n}(t)}(\overline{n}'(t))=\pi_{\overline{\nn}}(-\pi_\kk(X(t))^{u(t)^{-1}}).
\end{equation}
To simplify the right hand side  we use two facts. Firstly, we can replace the projection $\pi_\kk$ by the identity map  
because the kernel of $\pi_\kk$ is both invariant under the adjoint action of $\U$ and a subspace of the kernel of $\pi_{\overline{\nn}}$.
Secondly, the restriction of $\pi_{\overline{\nn}}$ to $H^{w}+\overline{\nn}$ amounts to substracting $H^{w}$.
 The conclusion is that (\ref{eq:sum-one}) becomes 
\[dR_{\overline{n}(t)}(\overline{n}'(t))=-\pi_\nn(X(t)^{u(t)^{-1}})=-\pi_\nn(B(t))=-B(t)+H^{w},\]
which together with (\ref{eq:Toda-sub-zero}) gives  $B'(t)=[-B(t)+H^{w},B(t)]$.
Finally, one has has
\begin{equation}\label{eq:Toda-pre-linear}
B'(t)=[-B(t)+H^{w},B(t)]=[H^{w},B(t)],
\end{equation}
which is the linear vector field in (\ref{eq:Toda-linear-coord}).
Its crical point and its stable and unstable manifolds are easily described.

\end{proof}

\begin{remark} (The two atlases and the action of $\A$)
 The last equality in (\ref{eq:Toda-pre-linear}) 
can be interpreted in the following conceptual manner: 
The adjoint action of $\overline{\N}$ on $H^{w}+\overline{\nn}$ is free so any vector at a point in $H^{w}+\overline{\nn}$
is the value of a unique fundamental vector field. The Iwasawa fiber is also stable by the adjoint action 
of $\U$, and thus one gets extra freedom in the choice of fundamental vector field. That freedom is used in (\ref{eq:Toda-pre-linear}) 
to trade a vector in $\overline{\nn}\subset \uu$ by a (constant!) vector in $\aa\subset \uu$. The conclusion is that
$\varphi_w:\cU_w\to H^{w}+\overline{\nn}$ intertwines the Toda vector field and the fundamental vector field 
of $H^{w}\in \aa$ for the adjoint action of $\A$ on $H^{w}+\overline{\nn}$. If we further pass from the affine to the vector space picture
$H^{w}+\overline{\nn}\to \overline{\nn}$, we obtain the fundamental vector field of $H^{w}\in \aa$ for the adjoint action of $\A$ on $\overline{\nn}$. 

The Bruhat coordinates on $\cL_{w}$ intertwine the fundamental vector fields for the left action of $\A$ on $\K/\M$ 
(defined via the diffeomorphism
$\K/\M\cong \G/\B$) and for the adjoint action of $\A^{w^{-1}}$ on $\overline{\nn}$  \cite[Lemma 7.1]{DKV}. In particular 
$\psi_{w^{-1}}:\cL_{w^{-1}}\to \overline{\nn}$ intertwines the the fundamental 
vector field of $H\in \aa$ for the left action on $\K/\M$ and the fundamental vector
field of $H^{w}\in \aa$ for the adjoint action of $\A$ on $\overline{\nn}$. 

This leads to diffeomorphisms $\psi_{\w^{-1}}^{-1}\circ \varphi_w:\cU_w\to \cL_{w^{-1}}$ which are  different from the diffeomorphisms $\tau_w$
in (\ref{eq:un-to-nu}). (The differential  of $\psi_{\w^{-1}}^{-1}\circ \varphi_w\circ \tau_w^{-1}$ at $H^{w^{-1}}$ is not the identity). 
\end{remark}

% 
% 
% 
% 
% \begin{theorem}\label{thm:lin-atlas}
%  Let $\gg$ be a non-compact real semisimple Lie algebra with fixed Iwasawa decomposition and let $\varphi_w:\cU_w\to H^\w+\overline{\nn}$, $w\in \W$
%  be the local coordinates in Definition \ref{def:atlas}.
% Then the diffeomorphism $\varphi_w$ transforms the Toda vector field $\cT$ into the linear constant coefficient vector field
%  \[X'=[H^{w},X],\quad X\in H^w+\overline{\nn},\]
%  whose stable and unstable manifolds at the critical point $H^w$ are the affine subspaces
%  \[H^w+\overline{\nn}(w),\quad H^w+\overline{\nn}(\overline{w}).\]
% \end{theorem}

Theorem \ref{thm:st-ust}, relating Bruhat and opposite Bruhat cells to unstable and stable manifolds of $\cT$, follows from
Theorem \ref{thm:lin-atlas} and the results in Section \ref{sec:atlas}.

\begin{proof}[Proof of Theorem \ref{thm:st-ust}]
By Proposition \ref{pro:atlas-Bruhat}, $\varphi_w(\cB_w)=H^w+\overline{\nn}(\overline{w})$. By Theorem \ref{thm:lin-atlas}, the right hand side is
the unstable manifold at $H^w$ of the image of $\cT$ by $\varphi_w$. Hence
the unstable manifold of $\cT$ at $H^w\in \cO$ is the Bruhat cell $\cB_w$. Likewise, the stable manifold at $w$
is  $\overline{\cB}_w$. The Morse-Smale property follows from the transverse
intersection of Bruhat cells and opposite Bruhat cells into Richardson cells (see e.g. \cite{R} or \cite[Lemma 4.2]{DKV}).
\end{proof}

\noindent {\bf Remark 1.} (Comparison with earlier results)

In \cite{SV} the equality between Bruhat cells and unstable manifolds for 
$\sl(n,\R)$ and $\sl(n,\C)$ is proved in two steps:
\begin{itemize}
 \item It is shown than for any invertible matrix  whose eigenvalues have different norm, left multiplication
defines a hyperbolic dymamical system with unstable manifolds (conjugated to) the Bruhat cells \cite[Theorem 1]{SV}; 
for determinant one matrices  
this is an immediate consequence of the earlier
infinitesimal results of \cite[Section 3]{DKV} which are valid for arbitrary non-compact semisimple Lie algebras (cf. Remark 3).
\item The factorization of integral curves of the Toda vector field \cite{RS,Sy} is used to deduce that for a given
regular hyperbolic element the same diagonal matrix 
is the common limit of a Toda trajectory through it and of the iterations 
of left multiplication by a related invertible invertible matrix  whose eigenvalues have different norm \cite[Pages 190 to 192 and Corollary 1]{SV}.
\end{itemize}

 In \cite{CSS1} the equality between Bruhat cells and unstable manifolds is proved first for $\sl(3,\R)$ and $\sl(4,\R)$ by direct computations,
and then for $\sl(n,\R)$ by an inductive argument; the case of rank 2 non-compact semisimple Lie algebras in \cite{CSS2,CSS3} is reduced to $\sl(n,\R)$ by
means of appropriate embeddings.

Both approaches rely on properties which seem specific of matrix calculus for the general/special linear group. 
% and their full extension to arbitrary non-compact semisimple Lie algebras is unclear.  
Moreover, what is addresed 
is the analysis of stable and unstable manifolds of the Toda vector field, and not its linearisation in appropriate coordinates
as described in Theorem \ref{thm:lin-atlas}.

\noindent {\bf Example 3.} We illustrate Theorem \ref{thm:lin-atlas} for the chart $\varphi_e^{-1}$ of $\cO\subset \sl(2,\R)$ computed
in Example 2.
The expression for the Toda vector field $\mathcal{T}(Y)=[Y,\pi_\mathfrak{k}Y]$,  for $Y$ a symmetric $2\times 2$ tracesless matrix is:
\begin{equation}\label{eq:Toda2x2}\mathcal{T}(Y)=\begin{pmatrix} 2b^2 & -2ab \\ -2ab & -2b^2\end{pmatrix},
\quad Y=\begin{pmatrix} a & b \\ b & -a\end{pmatrix}.
\end{equation}
In the domain of the chart corresponding to the identity element of the Weyl group $\varphi^{-1}_e:H+\overline{n}\to \cU_e$
the linear vector field coming from Theorem  \ref{thm:lin-atlas} is:
\[\begin{pmatrix} \lambda & 0 \\ x & -\lambda \end{pmatrix}\mapsto \begin{pmatrix} 0 & 0  \\ -2\lambda x & 0\end{pmatrix},\quad \lambda>0,\,x\in \R. \]
By (\ref{eq:chart}) its pullback by $\varphi_e^{-1}$ is

\[  \frac{-2\lambda x}{\left(1+\tfrac{x^2}{4\lambda^2}\right)^2}
\begin{pmatrix}
           -\tfrac{x}{2\lambda}(1+\tfrac{x^2}{4\lambda^2})-(\lambda-\tfrac{x^2}{4\lambda^2})\tfrac{x}{2\lambda^2} &
           (1+\tfrac{x^2}{4\lambda^2})-x\tfrac{x}{2\lambda^2}\\   (1+\tfrac{x^2}{4\lambda^2})-x\tfrac{x}{2\lambda^2}
          & \tfrac{x}{2\lambda}(1+\tfrac{x^2}{4\lambda^2})+(\lambda-\tfrac{x^2}{4\lambda^2})\tfrac{x}{2\lambda^2}
          \end{pmatrix}= \]
 \[         =\frac{-2\lambda x}{\left(1+\tfrac{x^2}{4\lambda^2}\right)^2}
%\begin{pmatrix}
%           -\tfrac{x}{\lambda} &   1-\tfrac{x^2}{4\lambda^2}\\   1-\tfrac{x^2}{4\lambda^2}  & \tfrac{x}{\lambda}
%          \end{pmatrix}
          =\left(1+\tfrac{x^2}{4\lambda^2}\right)^2\begin{pmatrix}
          2x^2 & -2x\left(\lambda-\tfrac{x^2}{4\lambda}\right)\\
         -2x\left(\lambda-\tfrac{x^2}{4\lambda}\right) & -2x^2
          \end{pmatrix}.
 \]

This vector field equals the result of replacing $Y$ by $\varphi_e^{-1}$ in (\ref{eq:Toda2x2}), which proves that $\varphi_e$ linearizes
the Toda vector field for $\sl(2,\R)$.

\section{Hessenberg-type manifolds}\label{sec:Hessenberg}

In this section we show how our atlas is adapted to the Hessenberg-type submanifolds of the manifold of real full flags $\cO$.

An (upper) Hessenberg matrix $M\in \M(n,\R)$ is a matrix all whose entries below the subdiagonal vanish. All upper triangular matrices are of
Hessenberg type. It is natural to consider other subspaces of matrices which contain the upper triangular matrices by generalizing
the Hessenberg condition on vanishing entries: we define a \emph{profile} $\p$ to be a subset of indices $\p\subset \{1,\dots,n\}^2$
which has the following properties.
\begin{enumerate}[(a)]
 \item $(i,j)\in {\bf\mathrm{p}}\Longrightarrow i>j$;\vskip .3cm
 \item  $(i,j)\in {\bf\mathrm{p}},\,\,  i\geq \tilde{i}>\tilde{j}\geq j
       \Longrightarrow (\tilde{i},\tilde{j})\in {\bf\mathrm{p}}$.
\end{enumerate}
Profiles inherit the partial order given by inclusion. The \emph{support} of a matrix $M$ is the smallest
profile  ${\bf\mathrm{p}}$ such that
$i>j\,\,\,\mathrm{and}\,\,\, M_{ij}=0\Longrightarrow (i,j)\notin {\bf\mathrm{p}}$.
The vector subspace of \emph{matrices with profile ${\bf\mathrm{p}}$} is
\[V_{\bf\mathrm{p}}=\{M\in \mathrm{M}(n,\R)\,|\, \mathrm{supp}(M)\subset {\bf\mathrm{p}}\}.\]
For  $\mathrm{p}=\{(n,1)\}$ we have $V_{\bf\mathrm{p}}=\mathrm{M}(n,\R)$.
Hessenberg matrices equal $V_p$ for $\p=\{(2,1),(3,2),\dots,(n,n-1)\}$. In general $V_\p$ are matrices whose lower
triangular entries avoiding certain square submatrices aligned along the subdiagonal are trivial.

The definition of
a profile is Lie theoretic and can be extended to any non-compact real semisimple Lie algebra $\gg$ \cite[Section 3]{MP}.
Upon a choice of root
(partial) order $(\Sigma,\preccurlyeq)$, the root space $\Sigma$ is the union of the positive and negative roots,
$\Sigma=\Sigma^+\cup \Sigma^-$. A \emph{profile} $\mathrm{p}$ for $\gg$ is a subset of $\Sigma$ which has the following properties.
 \begin{enumerate}[(a)]
 \item $\a \in {\bf\mathrm{p}}\Longrightarrow \a\in \Sigma^-$;\vskip .3cm
 \item  $   \a\in {\bf\mathrm{p}},\,\,
     \beta\in \Sigma^-,\,\,\a\preccurlyeq\ \beta  \Longrightarrow \beta\in {\bf\mathrm{p}}$.
\end{enumerate}
The \emph{support} of $X=\sum_{\alpha\in \Sigma} X_\alpha \in \gg$ is the smallest profile $\mathrm{p}$ such that
\[\beta\in \Sigma^-\,\,\,\mathrm{and}\,\,\,X_\beta=0 \Longrightarrow \beta\notin \mathrm{p}.\]
The vector subspace of \emph{elements with profile ${\bf\mathrm{p}}$} is
\begin{equation}\label{eq:profile}
V_{\bf\mathrm{p}}=\{X\in \gg\,|\, \mathrm{supp}(M)\subset {\bf\mathrm{p}}\}.
\end{equation}

\begin{definition}\label{def:isospectral}\quad
Let $\gg$ be a non-compact real semisimple Lie algebra and let $\p$ be a profile for $\gg$. The Hessenberg-type and non-compact Hessenberg-type
subsets with profile $\p$ are the intersections
\[\cO_\p=\cO\cap V_\p,\quad \cO^\G_\p=\cO^\G\cap V_\p,\]
where $\cO\subset \cO^\G$
are the adjoint $\K$ and $\G$ orbits of a regular element in the positive Weyl chamber.
\end{definition}

One can deduce from results in  \cite[Proposition 7.1]{MP}  that Hessenberg-type subsets are submanifolds 
(see Remark \ref{rem:Hessenberg-type}). Below we provide an alternative proof.

\begin{lemma}\label{pro:Hess-manifold}  Let $Y\in \cO$ and let $\mathrm{p}$ be a profile for $\gg$.
The $\K$-orbit of $Y$ and the vector subspace $V_\mathrm{p}$
  have transverse intersection. Therefore $\cO_\p$ and $\cO^\G_\p$
  are manifolds.
\end{lemma}
\begin{proof}
We may assume that $Y\in \mathrm{p}$. The tangent space of the $\K$-orbit at $Y$ is $[\kk,Y]$ so we
must show
\begin{equation}\label{eq:transversality-p}
\gg=[\kk,Y]+V_\p.
\end{equation}
Since $\gg=\kk\oplus \uu$, $Y\in V_p$ and the action of $\uu$ preserves $V_\p$, equality (\ref{eq:transversality-p})
is equivalent to
\begin{equation}\label{eq:transversality-p-full}
\gg=[\gg,Y]+V_\p.
\end{equation}
Since the subsets $\cU_w$ cover $\cO$ there exists $w\in \W$,  $k\in \cC,\,h\in \overline{\N}$ so that
$Y={(H^{w})}^{kh}$.
Equivalently, $Y={(H^w)}^{u\overline{n}}$, $u\in \U,\,\overline{n}\in \overline{\N}$.
Upon conjugating  (\ref{eq:transversality-p-full}) by $u^{-1}$ and using that the action of $\U$ preserves $V_\p$, we are led to 
prove the equality 
\[\gg=[\gg,{(H^w)}^{\overline{n}}]+V_\p.\]
Now
\[[\gg,{(H^w)}^{\overline{n}}]+V_\p\supset [\overline{\nn},H^w]+V_\p=\overline{\nn}+ V_\p=\gg,\]
which proves the equality.
\end{proof}

We are ready to collect the results announced in Theorem \ref{thm:atlas} in the introduction.
\begin{proof}[Proof of Theorem \ref{thm:atlas}]
Item (i) is the content of Lemma \ref{lem:cover} and Definition \ref{def:atlas}. Item (ii) is the content of Theorem \ref{thm:lin-atlas}. We observe that
$\cT$ on $\cU_w$ is complete because it corresponds to a linear vector field on a whole Euclidean space.

Let us denote $\cU_{w,\p}=\cU_w\cap V_\p$, $\overline{\nn}_\p=\overline{\nn}\cap V_\p$.
By construction the subsets
$\cU_{w,\p}$, $w\in W$, define an open cover of $V_\p$. Both the local coordinates $\varphi_w:\cU_w\to H^{w}+\overline{\nn}$ and its inverse 
diffeomorphism are defined by conjugating with appropiate elements in $\U$. Because the action of $\U$ preserves 
we conclude that $\varphi_w$ restricts to $\cU_{w,\p}$ to a diffeomorphism onto $H^w+\overline{\nn}_\p$.
Therefore the atlas $(\cU_w,\varphi_w)$, $w\in \W$, is adapted to $\cO_\p$
and the image of $\cU_{w,\pp}$ is a full affine subspace of $\overline{\nn}$.

Since $V_\p$ is preserved by the action of $\U$, the Toda vector field $\cT$ is tangent to the manifolds $\cO_\p$ (and $\cO^\G$). Hence
$\varphi_w$ takes $\cT$ on $\cU_{w,\pp}$ to a linear vector field on $H^{w}+\overline{\nn}_\p$. In other words,
the atlas also linearizes the restriction of $\cT$ to the Hessenberg-type submanifolds $\cO_\p\subset \cO$.
\end{proof}

\noindent {\bf Remark 2.}  Hessenberg-type subspaces for $\gg$ are $\uu$-modules and the two notions match if and only if $\gg$
has 1-dimensional root spaces. The results in this Section and in Section \ref{sec:flow} for Hessenberg-type submanifolds remain
valid if we replace Hessenberg-type subspaces by  $\uu$-modules.

\begin{remark}\label{rem:Hessenberg-type} (Hessenberg-type submanifolds versus real Hessenberg varieties)
% \label{rem:Bruhat} (Comparison with the Bruhat open cover)
There is a second family of subsets that one can associate to a profile $\p$ and a regular element $H\in \aa$. According to \cite{MP}, 
the real Hessenberg variety 
is  the
subset of $\K/\M$
\[\mathrm{Hess}_\p(H)=\{kM\in \K/\M\,|\, H^{(k^{-1})}\in V_\p\}.\]
Real Hessenberg varieties are smooth \cite[Proposition 7.1]{MP}. One can deduce from that the smoothness of Hessenberg-type subsets in two
different ways. Firstly, the smoothness of $\mathrm{Hess}_\p(H)$ and of $\cO_\p\subset \cO\cong \K/\M$ is equivalent to the smoothness of their respective
preimages 
in $\K$ by the quotient map $\K\to \K/\M$, and the inversion map on $\K$ establishes a bijection between the preimages. Secondly, one can 
relate the subsets without passing to preimages:  the diffeomorphism $\tau_w:\cU_w\to \cL_{w^{-1}}$ induces
a bijection from $\cU_{w,\p}$ to $\mathrm{Hess}_\p(H)\cap \cL_{w^{-1}}$. The proof of the smoothness of $\mathrm{Hess}_\p(H)$ in  
\cite[Proposition 7.1]{MP} is done using the cover $\cL_w$, $w\in W$, and Lemma \ref{pro:Hess-manifold} uses the cover $\cU_w$, $w\in W$. 
One can see that both proofs are related by $\tau_w$.

Finally, we point out that the manifolds $\cO_\p$ and $\mathrm{Hess}_\p(H)$ may not be homeomorphic \cite[Example 3.13]{AB}. 
\end{remark}

\section{The contracting flow on the open Hessenberg-type manifolds}\label{sec:flow}

The Iwasawa fibration makes $\cO^{G}\to \cO$ into a vector bundle. However, this vector bundle structure 
does not induce a vector bundle structure on the non-compact Hessenberg-type manifold (over the compact one). 
The rank of the
intersection of the linear Iwasawa fibers with $V_\p$ varies. Equivalently, the Euler vector field which 
encodes the vector bundle structure on $\cO^{G}\to \cO$ is not tangent to $\cO^\G_\p$.

In this section we describe a natural fibration from the noncompact to the compact Hessenberg-type manifolds:
\begin{equation}\label{eq:Hess-submersion}
 \cO^\G_\p\to \cO_\p.
\end{equation}
It will be induced by the flow of a vector field on $\cO^\G$ which is tangent to $\cO^\G_\p$.

% This also adds to Remark 2: The Iwasawa factorization $\G=\K\A\N$ induces a vector bundle
% structure $\cO^G\to \cO$. The vector bundle structure is compatible with arbitrary hyperbolic orbits \cite[Section 2]{DKV}. The
% Iwasawa factorization restricts to the center of a non-regular elements  $X\in \aa$. Therefore the corresponding orbit
% inherits a vector bundle structure $X^\G\to X^K$ and the bundle structures are compatible with the natural submersion
% $\cO^\G\to X^\G$. Equivalently the Euler vector fields which encode the vector bundle structures are related by the submersion.
% However, the vector bundle structure  $\cO^G\to \cO$ does not restrict to the open Hessenberg-type manifold $\cO^\G_\p$. The rank of the
% intersection of the linear fibers with $V_\p$ varies. Equivalently, the Euler vector field on $\cO^\G$ is no longer tangent to $\cO^\G_\p$.

We define the vector field on $\gg$
\begin{equation}\label{eq:contracting}
\mathcal{S}(X)=[X,\pi_\uu [X,-\theta X]],
\end{equation}
where $\pi_\uu$ is the second projection associated to the decomposition $\gg=\kk\oplus \uu$.
For $\gg=\mathfrak{sl}(n,\R)$ this vector field reads (cf.  \cite[Equation (1.2)]{AE})
\[\mathcal{S}(X)=[X,\pi_\uu [X,X^T]],\]
where $X^{T}$ is the transpose matrix of $X$.

Here is an account of its basic properties.
\begin{lemma}\label{lem:hypbasic}
The vector field $\mathcal{S}$ in (\ref{eq:contracting}) has the following properties.
\begin{enumerate}
\item  It is tangent to $\cO_\p^\G$,  where $\p$ is any profile for $\gg$ 
 \item The norm square $||\cdot ||^2$ (with respect to $B_\theta (\cdot,\cdot)=-\langle \cdot,\theta \cdot\rangle $)
 is monotone decreasing on the trajectories of $\mathcal{S}$.
 \item The zeroes of $\mathcal{S}$ are the elements 
 $X\in \gg$ such that $[X,\theta X]=0$, which will be called normal.
 \item The intersection of normal elements with regular hyperbolic elements is the subset $\pp^r$ of regular elements of $\pp$.
\end{enumerate}

\end{lemma}

\begin{proof}
Since $\mathcal{S}$ is in Lax form it is tangent to $\cO^\G$. Since  $\pi_u [X,-\theta X]\in \uu\subset \gg$, if $X\in V_\p$ then
$\mathcal{S}(X)\in V_\p$, which proves (1).

By differentiating the norm square along a trajectory $X(t)$ we obtain
\[\begin{split}
\frac{d}{dt}||X(t)||^2 &=2B_\theta(X'(t),X(t))=-2\langle [X,\pi_\uu[X,-\theta X]],-\theta X\rangle=\\
&=2\langle \pi_\uu[X,-\theta X],[X,-\theta X]\rangle= -2B_\theta([X,-\theta X],[X,-\theta X]).\\
  \end{split}
\]
In the last equality we used that $[X,-\theta X]\in \pp$ and that $\kk$ and $\pp$ are $B_\theta$-orthogonal so that  adding $\pi_\kk$ in the first entry
does not modify the result of the inner product. This proves assertion (2).

Normal elements are clearly zeroes of $\mathcal{S}$. Conversely, if $X$ is not normal,  by item (2)
the norm square decreases along the trajectory
through $X$  around $X$.

Let $X$ be a regular hyperbolic normal element. That $X$ be normal means that $\theta X$ belongs to the centralizer $\gg_X$ of $X$ in $\gg$.
We claim that the Cartan involution preserves $\gg_X$. To
prove that we use that $\gg_X$ is a Cartan subalgebra: if $Y\in \gg_X$, since $\theta X\in \gg_X$ we must have
\[[Y,\theta X]=0\Rightarrow [\theta Y,X]=0.\]
Any $\theta$-stable Cartan subalgebra decomposes into direct summands in $\kk$ and $\pp$.
Therefore $X=X_e+X_h+0$, $X_e\in \kk$, $X_h\in \pp$, is the decomposition of $X$ into commuting elliptic, hyperbolic and nilpotent
summands \cite[Section 2]{S}.
But because $X$ is hyperbolic $X_e$ must be trivial, and $X=X_h\in \pp$
($X_h$ is regular and therefore $\gg_X$ must be maximally non-compact: $\gg_X\cap \pp=\aa$).
\end{proof}

By item (1) in Lemma \ref{lem:hypbasic} the vector field $\mathcal{S}$ is tangent to the adjoint orbit
$\cO^\G$. By items (3) and (4) in Lemma \ref{lem:hypbasic}
the restriction of $\mathcal{S}$ to $\cO^\G$ vanishes on  $\cO$.
Therefore  the restriction of $\mathcal{S}$ to $\cO^G$ has an  intrinsic linearization along the submanifold $\cO$:
$\nabla\mathcal{S}: T\cO^\G|_{\cO}\to T\cO^\G|_{\cO}$.
The explicit formula at $X\in \cO$ is
\[(\nabla\mathcal{S})_X Y=-[X,\pi_\uu[X,Y+\theta Y]], \quad Y\in T_X\cO^\G.\]

\begin{proposition}\label{pro:negeigen} Let $H$ be a regular element in the positive Weyl chamber and $k\in \K$.
Then
\begin{equation}\label{eq:eigen-subbundles}T_{H^k}\cO^\G=
\pp\oplus\mathrm{Im}(\nabla\mathcal{S})_{H^k}=\pp\oplus \sum_{\alpha\in \Sigma^+}\mathrm{Im}\left((\nabla\mathcal{S}_{H^{k}})|_{\gg_\alpha^k}\right)
\end{equation}
is a as direct sum decomposition in eigenspaces of $(\nabla\mathcal{S})_{H^k}$ for the zero eigenvalue and strictly negative eigenvalues.
\end{proposition}
\begin{proof}
Given a linear map, the image of a vector is an eigenvector exactly when the linear map restricted to the line spanned by the vector is  
a multiple of a projection.
We want to apply that to $(\nabla\mathcal{S})_{H^{k}}$. 

In the computation below of the square of $\nabla\mathcal{S}$ at $X\in \cO$ applied to $Y\in T_X\cO^\G$ we use that the
decomposition  $\mathrm{I}=\pi_\kk+\pi_\uu$ implies the equality $\theta\pi_\uu=\theta-\pi_\kk$.
\[\begin{split}
\left((\nabla\mathcal{S})_X\right)^{2} Y &=-[X,\pi_\uu[X,-[X,\pi_\uu[X,Y+\theta Y]]+\theta(-[X,\pi_\uu[X,Y+\theta Y]])]]=\\
&=[X,\pi_\uu[X,[X,\pi_\uu[X,Y+\theta Y]]+[\theta X,\theta-\pi_\kk[X,Y+\theta Y]])]]=\\
&=[X,\pi_\uu[X,[X,\pi_\uu[X,Y+\theta Y]]+[X,[X,Y+\theta Y]]+[X,\pi_\kk[X,Y+\theta Y]])]]=\\
&=2[X,\pi_\uu[X,[X,[X,Y+\theta Y]]]].
\end{split}
\]
If $H$ is a regular element in the positive Weyl chamber, $X=H^k$, $k\in \K$,
and  $Y\in \gg^k_\alpha$, $\alpha\in \Sigma^+$,
then using $\theta \gg_\alpha^k=\gg_{-\alpha}^k$ we get
\[\left((\nabla\mathcal{S})_{H^{k}}\right)^{2} Y=2[X,\pi_\uu[X,[X,[X,Y+\theta Y]]]]=-2\alpha(H)^2(\nabla\mathcal{S})_{H^k}Y.\]
Therefore  $\mathrm{Im}\left((\nabla\mathcal{S}_{H^{k}})|_{\gg_\alpha^k}\right)$ is a subspace
of the eigenspace\footnote{We are abusing terminology in the statement of the proposition because the eigenspaces there are subspaces on which 
the linear map is a multiple of the identity, since different positive roots may have 
the same value on $H$.} of $(\nabla\mathcal{S})_{H^{k}}$ for the eigenvalue $-2\alpha(H)^2$. 
Since  $\pp$ is in the kernel of $(\nabla\mathcal{S})_{H^{k}}$ a dimension count implies
that (\ref{eq:eigen-subbundles}) is a direct sum decomposition.
\end{proof}

\begin{theorem}\label{thm:Hess-submersion}
The flow of $\mathcal{S}$ induces a locally trivial fibration structure
\begin{equation}\label{eq:submersion}
 \cO^\G\to \cO
\end{equation}
with fibre diffeomorphic to a Euclidean space.

Moreover, if $\p$ is any profile the fibration (\ref{eq:submersion}) restricts to a (pullback!) fibration
\begin{equation}\label{eq:Hess-submersion-pullback}
\cO_\p^\G\to \cO_\p=(\cO^\G\to \cO)|_{\cO_\p}
\end{equation}
whose fiber over each intersection point $H^w\in \cO_\p\cap \aa$ is $H^w+\nn$.
 \end{theorem}
\begin{proof}
The norm square  $||\cdot ||^2$ is a proper function on $\gg$.
This, together with items (2) in Lemma \ref{lem:hypbasic}, implies
that $\mathcal{S}$ is complete on $\gg$.  Since $\cO^\G\subset \gg$ is made of semisimple elements it is a closed submanifold
\cite[Proposition 10.1]{BHC}. This, together with items (2) and (3) in Lemma \ref{lem:hypbasic}, implies
that the limit set of a trajectory in $\cO^\G$ must be contained in $\cO$.
Proposition \ref{pro:negeigen} and elementary O.D.E. theory imply that the flow of $\mathcal{S}$ on  $\cO^\G$
is $k$-normally hyperbolic relative to $\cO$, for all $k\in \N$.
By the unstable manifold theorem \cite{HPS,F1} (the global version using completeness of $\mathcal{S}$) there is a canonically defined surjective
submersion
$\cO^\G\to \cO$ where each fibre are the trajectories with limit set the corresponding base point. The estimates in \cite[Section IA]{F2}
for the projection map to be
smooth are satisfied. The fibration has vertical
bundle at $\cO$ exactly $\mathrm{Im}(\nabla\mathcal{S})$. There exist (non-canonical) fibre bundle diffeomorphisms from
$\cO^\G$ to $\mathrm{Im}(\nabla\mathcal{S})$. It also follows that $\cO^\G\to \cO$
is a locally trivial fibration (alternatively, one can use that this is always the case for surjective submersion with fibers diffeomorphic
to Euclidean space \cite[Corollary 31]{Me}).

Let $\p$ be a profile for $\gg$.  By item (1) in Lemma \ref{lem:hypbasic} the vector field $\mathcal{S}$ is tangent to
$\cO_\p^\G$. Let us denote its restriction by $\mathcal{S}|_{\cO_\p^\G}$. 
We repeat the arguments above but now for $\mathcal{S}|_{\cO_\p^\G}$ to deduce that 
it is complete with
trajectories having non-empty limit set contained in its zero set $\cO_\p$. Its linearization there 
\[\nabla(\mathcal{S}|_{\cO_\pp^\G}):T\cO_\p^\G|_{\cO_\p}\to T\cO_\p^\G|_{\cO_\p}\]
is simply the restriction to  
 $T\cO_\p^\G|_{\cO_\p}$ of the linearization  $\nabla\mathcal{S}:T\cO^\G|_{\cO}\to T\cO^\G|_{\cO}$. Therefore the kernel of $\nabla(\mathcal{S}|_{\cO_\pp^\G})$ 
must be  
\[\ker\left(\nabla\mathcal{S}\right)\cap T\cO_\p^\G|_{\cO_\p}=\pp\cap T\cO_\p^\G|_{\cO_\p}=T\cO_\p.\]
By the rank Theorem the image of $\nabla(\mathcal{S}|_{\cO_\pp^\G})$ is a bundle whose rank is the dimension of $\nn$ (this is the codimension
of $\cO_\p$ in $\cO_\p^\G$, as seen in \cite[Section 7]{MP} or in the proof of Lemma \ref{pro:Hess-manifold}). Since the dimension of
$\nn$ equals the
rank of $\nabla\mathcal{S}:T\cO^\G|_{\cO}\to T\cO^\G|_{\cO}$, we have
\[\mathrm{Im}\left(\nabla(\mathcal{S}|_{\cO_\pp^\G})\right)=\mathrm{Im}(\nabla\mathcal{S})|_{\cO_\p}.\]
Therefore $\mathcal{S}|_{\cO_\pp^\G}$ is also normally hyperbolic relative to $\cO_\p$ and the (stable) fibers of $\cO_\p^\G$ are full
(stable) fibers of $\cO^\G$, which proves (\ref{eq:Hess-submersion-pullback}).

By (\ref{eq:contracting}) $\mathcal{S}$ is tangent to the orbits of the adjoint action of $\U$. In particular, it is tangent to  
$H^w+\nn$. Because the dimension of the fibers of $\cO_\p^{\G}\to \cO_\p$ equals the dimension of $\nn$, the affine subspaces $H^w+\nn$, $w\in W$, must be fibers.
\end{proof}

Theorem \ref{thm:Hess-submersion} for $\gg=\sl(n,\R)$ gives the flow on traceless matrices onto symmetric ones announced in the introduction.
\begin{proof}[Proof of Theorem \ref{thm:Hess-submersion-sl}]
Clearly the zero set of $\mathcal{S}$ on $\sl(n,\R)$ are the normal matrices: this is item (3) in Lemma \ref{lem:hypbasic}.
Completeness is shown in the first part of the proof of Theorem \ref{thm:Hess-submersion}. Item (1) in Lemma \ref{lem:hypbasic} states that the flow preserves the spectrum
and the Hessenberg-type subspaces  (strictly speaking the lemma would only apply to traceless matrices with simple real spectrum,
but the Lax form of $\mathcal{S}$ implies the result for all traceless matrices).

The first part of item (i) in Theorem \ref{thm:Hess-submersion-sl} corresponds to the first lines in the proof of Theorem \ref{thm:Hess-submersion}.
The normal hyperbolicity properties of $\mathcal{S}$ on the open subset of matrices with simple real spectrum follow from the proof of Proposition
\ref{pro:negeigen}. The statement of the Proposition is for individual hyperbolic orbits, but it is equally true the collection
of all regular hyperbolic orbits. Indeed, they fit into a submanifold of $\gg$
foliated by regular hyperbolic orbits for which normal regular hyperbolic elements are submanifold transverse to the foliation by regular hyperbolic orbits.

Items (iii) in Theorem \ref{thm:Hess-submersion-sl} is equation (\ref{eq:submersion}) in Theorem \ref{thm:Hess-submersion}.
Item (ii) asserts the submersion property for all
regular hyperbolic orbits, which also holds for the reasons stated in the paragraph above.

Item (iv) is the equality in equation (\ref{eq:Hess-submersion-pullback}) between the restriction of the submersion to the
Hessenberg-type subspace $V_\p$, and the pullback of the
submersion over the intersection of the base with $V_\p$. In particular, the only flow lines converging to
diagonal matrices -- symmetric upper triangular matrices-- are those made of upper triangular matrices, the
Hessenberg-type subspace for the empty profile.
\end{proof}

\noindent {\bf Example 4.} We write the vector field (\ref{eq:contracting}) for  $\sl(2,\R)$
and discuss Theorem  \ref{thm:Hess-submersion} for the hyperboloids.
Identify  $\R^3$ with $\sl(2,\R)$ by mapping the canonical basis to
\[\begin{pmatrix}
    1 & 0 \\ 0 & -1
  \end{pmatrix},\quad
  \begin{pmatrix}
    0 & 1 \\ 1 & 0
  \end{pmatrix}, \quad
   \begin{pmatrix}
    0 & -1 \\ 1 & 0
  \end{pmatrix}.\]
In Euclidean space coordinates the cubic vector field  (\ref{eq:contracting}) is
\begin{equation}\label{eq:contracting-2x2}
\mathcal{S}(x,y,z)=(-2zx(y+z),2zx^2-2z^2y,-2zx^2-2zy^2).
\end{equation}
It is tangent to the level sets of $x^2+y^2-z^2-1$ which correspond to the adjoint orbits (or union of them) and its stationary
points contain the plane $z=0$ which corresponds to traceless symmetric matrices. By
Theorem \ref{thm:Hess-submersion}, each stable point $(x,y,0) \ne 0$ is the limit point of two trajectories
of pairs of flow lines. These three orbits of the vector field (\ref{eq:contracting-2x2}) fit into a fiber of a locally trivial fibration of the hyperboloid over
the circle (the manifold of real full flags) through the point $(x,y,0)$.

We could not find and explicit expression neither for the trajectories
of the vector field (\ref{eq:contracting-2x2}) nor for the submersion it induces on hyperboloids. The explicit information that
Theorem (\ref{thm:Hess-submersion}) provides is that 
\begin{itemize}
 \item over $(x,0,0)$, which corresponds to a diagonal matrix, the fiber is the affine line tangent to $(0,-x,x)$, which
corresponds to an upper
triangular matrix;
\item the vector field (\ref{eq:contracting-2x2}) is tangent to the plane $x=0$ which corresponds to upper triangular matrices. Therefore 
the  hyperbolas on that plane are also fibers
of the submersion.
\end{itemize}

We can however compare the vertical tangent bundle of the submersion
at the circle with that of the Iwasawa fibration. The latter has as frame the vector field
\[(y,-x,\sqrt{x^2+y^2}),\quad x^2+y^2=\lambda^2,\,\lambda>0.\]
Following Proposition \ref{pro:negeigen}, to compute a frame for the former vertical tangent bundle we have to apply
the linearization of (\ref{eq:contracting-2x2}) at points of the circle, and then apply it to the Iwasawa frame.
This results into the (suitably rescaled) frame
\[(xy,-x^2,x^2+y^2),\quad x^2+y^2=\lambda^2,\,\lambda>0.\]

\end{document}